%% file: resultantArxv1.tex
\documentclass{amsart}
  \input amssymb.sty

\input tropMacro.tex

\def\pSkip{\vskip 1.5mm \noindent}

\newcommand{\tsqrt}[1]{\sqrt[{\operatorname{trop}}]{#1}}

\newcommand{\etype}[1]{\renewcommand{\labelenumi}{(#1{enumi})}}
\def\eroman{\etype{\roman}}

\newcommand{\Det}[1]{ \left|{#1}\right|}

\def\FunR{\operatorname{Fun} (R^{(n)},R)}
\def\CFunR{\operatorname{CFun} (R^{(n)},R)}
\def\CFunF{\operatorname{CFun} (F^{(n)},F)}
\def\CFunFF1{\operatorname{CFun} (F,F)}

\def\rstf{{\operatorname{l.s.}}}
\def\lstf{{\operatorname{r.s.}}}

\def\fl{f^\lstf}
\def\fr{f^\rstf}
\def\gl{g^\lstf}
\def\grt{g^\rstf}

\def\frb{\phi}

\newcommand{\spol}[2]{{#1}_{[#2]}}

\def\oset{W}

\def\osetb{U}

\def\gr{\Gm}
\def\grg{\gr^\nu}

\newcommand{\grgTf}[1]{\grg_{#1;\tT}}
\def\tp{\hat p}
\def\tq{\hat q}

\def\Fun{\operatorname{Fun}}

\def\tGz{\mathcal G_0}
\def\tTz{\mathcal T_0}

\def\tan{{\operatorname{tan}}}
\def\tng{{\operatorname{tan}}}

\def\ntng{{\operatorname{intan}}}
\def\ft{f^{\tng}}
\def\fm{f^{\ntng}}

\def\({\left(}
\def\){\right)}

\def\permanent{supertropical determinant}

\def\epsc{\in_{\operatorname{ir-com}}}
\newcommand{\tGinf}{\tG_{0}}

\def\res{\Re}

\def\tGz{\mathcal G_0}

\def\ghost{\text{ghost}}

\def\ldeg{\underline{\deg}}

\def\a{\alpha}
\def\la{\lambda}
\def\sig{\sigma}

\def\one{\mathbb{1}}
\def\zero{\mathbb{0}}

\def\sml{{\operatorname{sml}}}

\def\dom{{\operatorname{dom}}}
\def\deg{{\operatorname{deg}}}

\def\ra{a}
\def\bfa{\textbf{\ra}}
\def\rb{b}
\def\bfb{\textbf{\rb}}

\def\bfi{ \textbf{i}}
\def\bfj{\textbf{j}}

\newcommand{\dtr}[1]{\left|#1\right|}

\def\rzero{\zero_R}
\def\fone{\one_F}
\def\fzero{\zero_F}

\def\la{\lambda}

\newtheorem{thm}[theorem]{Theorem}
\newtheorem*{thm*}{Theorem}
\newtheorem*{dig*}{Digression}
\newtheorem{lem}[theorem]{Lemma}
\newtheorem{rem}[theorem]{Remark}
\newtheorem{prop*}{Proposition}

\newtheorem{prop}[theorem]{Proposition}
\newtheorem{defn}[theorem]{Definition}
\newtheorem*{examp*}{Example}
\newtheorem*{examples*}{Examples}
\newtheorem*{remark*}{Remark}

\newtheorem*{defn*}{Definition}
\newtheorem*{note*}{Note}

\begin{document}
\title[Supertropical Polynomials and Resultants] {Supertropical Polynomials and Resultants}

\author{Zur Izhakian}\thanks{The first author has been supported by the
Chateaubriand scientific post-doctorate fellowships, Ministry of
Science, French Government, 2007-2008}
\thanks{The second author is supported in part by the
Israel Science Foundation, grant 1178/06.}
\address{Department of Mathematics, Bar-Ilan University, Ramat-Gan 52900,
Israel} \address{ \vskip -6mm CNRS et Universit�e Denis Diderot
(Paris 7), 175, rue du Chevaleret 75013 Paris, France}
\email{zzur@math.biu.ac.il, zzur@post.tau.ac.il}
\author{Louis Rowen}
\address{Department of Mathematics, Bar-Ilan University, Ramat-Gan 52900,
Israel} \email{rowen@macs.biu.ac.il}


\date{\today}



\begin{abstract} This paper, a continuation of \cite{IzhakianRowen2007SuperTropical}, involves a closer study of polynomials
of supertropical semirings and their version of tropical geometry
 in which we introduce the concept of
relatively prime polynomials and resultants,  with the aid of some
topology. Polynomials in one indeterminant are seen to be
relatively prime iff they do not have a common tangible root, iff
their resultant is tangible. The Frobenius property yields a
morphism of supertropical varieties; this leads to a supertropical
version of B\'{e}zout's theorem. Also, a supertropical variant of
factorization is introduced which yields a more comprehensive
 version of Hilbert's Nullstellensatz than the one given in \cite{IzhakianRowen2007SuperTropical}.
\end{abstract}

\maketitle




\section{Introduction and review}
\numberwithin{equation}{section}

The supertropical algebra, a cover of the max-plus algebra,
explored in   \cite{zur05SetNullstellensatz},
\cite{IzhakianRowen2007SuperTropical}, was  designed to provide a
more comprehensive algebraic theory underlying tropical geometry.
The abstract foundations of supertropical algebra, including
polynomials over supertropical semifields,  are given in
\cite{IzhakianRowen2007SuperTropical}. The corresponding matrix
theory is explored in \cite{IzhakianRowen2008Matrices}, and this
paper is a continuation, exploring the resultant of supertropical
polynomials in terms of matrices, and the ensuing applications to
the resultant. The tropical resultant has already been studied by
Sturmfels ~\cite{Sturmfels94,SturmfelsSolving},
 Dickenstein, Feichtner,  and  Sturmfels~
\cite{dickensteinDiscreminants}, and
Tabera~\cite{TaberaResultant}, but our purely algebraic approach
is quite different,  leading to a tropical version of B\'{e}zout's
Theorem (Theorem \ref{thm:Bezout}).

Since this paper deals mainly with polynomials and their roots, it
could be viewed as a continuation of
\cite{IzhakianRowen2007SuperTropical}, although we explicitly
state those results that we need.
 We briefly review the underlying notions. The underlying
structure is a
 \textbf{semiring with ghosts}, which we recall is a
triple $(R,\tG,\nu),$ where $R$ is a semiring with zero element
$\zero_R$ (often identified in the examples with $-\infty$, as
indicated below), and $\tGinf = \tG \cup \{ \zero_R \} $ is a
semiring  ideal, called the \textbf{ghost ideal}, together with an
idempotent semiring homomorphism
$$\nu : R \ \To \ \tG\cup \{\zero_R\}$$  called the \textbf{ghost map}, i.e.,
which preserves multiplication as well as addition. We write
$a^{\nu }$ for $\nu(a)$, called the $\nu$-\textbf{value}  of $a$.
Two elements $a$ and $b$ in $R$ are said to be
\textbf{$\nu$-matched} if they have the same $\nu$-value; we say
that $a$ \textbf{dominates} $b$ if $a^\nu \ge b^\nu$. Two vectors
are \textbf{$\nu$-matched} if their corresponding entries are
 $\nu$-matched.

\begin{note}\label{supertr} Throughout this paper, we also assume the key property called \textbf{supertropicality}:
$$a+b   =  a^{\nu } \quad \text{if}\quad a^{\nu
} =
 b^{\nu}.$$
In particular, $ a+a=  a^\nu, \ \forall a\in R  $.
\end{note}

A \textbf{supertropical semiring}  has the extra structure that
$\tG$ is ordered, and satisfies the property called
\textbf{bipotence}: $a+b = a$ whenever $ a^{\nu }
> b^{\nu }.$

  A \textbf{supertropical domain} is a
supertropical semiring for which $\tT(R) = R\setminus \tGinf$ is a
 monoid,  called the set of \textbf{tangible elements} (denoted as $\tT$ when $R$ is unambiguous),
such that the map $\nu _\tT : \tT \to \tG$ (defined as the
restriction from $\nu$ to $\tT$) is onto. We write $\tTz$ for $\tT
\cup \{ \rzero\}$.
 We also define a \textbf{supertropical semifield}
to be a commutative supertropical domain $(R, \tG, \nu)$ for which
$\tT$ is a group; in other words, every  tangible element of $R$
is invertible. Thus, $\tG$ is also a (multiplicative) group.
Since any strictly ordered commutative semigroup has an ordered
Abelian group of fractions, one can often reduce from the case of
a (commutative) supertropical domain to that of a supertropical
semifield.

When studying a supertropical domain $R$, it is convenient to
define an inverse function $\hat \nu: R \to \tT,$ which is a
retract of $\nu$ in the sense that $\hat \nu$ is $1_{\tTz}$ on
$\tTz$, and writing $\hat a$ for $\hat \nu (a)$, we have $(\hat a)
^\nu= a^\nu$ for any $a \in R.$ (When $\nu_\tT$ is~1:1, we take
$\hat \nu$ to be $\nu_\tT^{-1}$ on $\tG$. In general, the function
$\hat \nu$ need not be uniquely defined if $\nu_{\tT}$ is not
1:1.)

The following natural topology is very useful in dealing with
certain delicate issues.

\begin{definition}\label{def:topology}   For any supertropical domain   $R= (R,\tG,\nu)$,
we define the $\nu$-\textbf{topology} to have a base of open sets
of the form $$\oset  _{\a, \beta} = \{ a \in R: \a^\nu <  a^\nu <
\beta^\nu \} \quad \text{ and } \quad \oset  _{\a, \beta; \tT} =
\{ a \in \tT: \a^\nu < a^\nu < \beta^\nu \},  \quad
\text{where}\quad \a^\nu, \beta^\nu \in \tGz.$$ We call such sets
\textbf{open intervals} and \textbf{tangible open intervals},
respectively. We say that $R$ is \textbf{connected} if each open
interval cannot be written as the union of two nonempty disjoint
intervals.

$R^{(n)}$ is endowed with the product topology induced by the
 $\nu$-topology on $R$.
\end{definition}

\begin{rem} $ $ \eroman
\begin{enumerate}
    \item  Clearly $a^{\nu} $ is in the closure of $\{a\}$, since any
open interval containing
 $a^{\nu} $  also contains $a$. \pSkip

    \item The $\nu$-topology restricts to a topology on $\tT$, whose
base is the set of tangible intervals. \pSkip
    \item We often will assume  that $R$ (and thus $\tT$) is divisibly closed,  by passing to
    the  divisible closure $\{ \frac a n : a \in R,\ n
\in \Net \}$; see \cite[Section
3.4]{IzhakianRowen2007SuperTropical} for details. \pSkip
\end{enumerate}

\end{rem}

\subsection{The function semiring} Our main
connection from supertropical algebra to   geometry comes from
supertropical functions, which we view in the following
supertropical setting:

\begin{definition}\label{ghost0}  $\FunR$ denotes the set of functions from
$R^{(n)}$ to~ $R$. A function $f \in\FunR$ is said to be
\textbf{ghost} if
$$f(a_1, \dots, a_n) \in \tGz$$
 for every  $a_1, \dots, a_n
\in R$; a function
 $f\in \FunR$ is  called \textbf{tangible} if $$f(J) \not \subseteq \tGz$$
  for every
nonempty open set $J$ of $R^{(n)}$ with respect to the product
topology induced by Definition~\ref{def:topology}.

 $\CFunR$ consists of the   sub-semiring comprised
of functions in the semiring  $\FunR$ which are continuous with
respect to the $\nu$-topology.
\end{definition}

\begin{rem} $\FunR$ has the ghost map $\nu$ given by defining $f^\nu(a) = f(a)^\nu.$
Thus, $\FunR$ is a semiring with ghosts, satisfying
supertropicality, although $\FunR$ is not a supertropical semiring
since bipotence fails.
\end{rem}

\begin{rem} The product $fg$ of tangible
 functions is also tangible. (Indeed, by definition, for
any open interval~$\oset  _1$, $f(\oset  _1)$ is not ghost, so
therefore $\oset _2 = \{ a \in \oset  _1: f(a)\in \tT\}$ is a
nonempty open set. By definition, $g(\oset  _2)$ is not ghost, and
thus $fg(\oset  _2)$ is not ghost.)
\end{rem}

\begin{definition}\label{dist} Functions $f_1, \dots, f_m \in \CFunR$ are
$\nu$-\textbf{distinct} on an open set $\oset $ if there is a
nonempty dense open set $\oset ' \subseteq \oset $ on which
$f_i(\bfa)^\nu \ne f_j(\bfa)^\nu$ for all $i\ne j$ and all
$\bfa\in \oset '$.
\end{definition}

\begin{rem} To satisfy Definition \ref{dist}, it is enough to
find dense $\oset _{ij}\subseteq \oset $  for each $i \ne j$, such
that  $f_i(\bfa)^\nu \ne f_j(\bfa)^\nu$ for all $\bfa\in \oset
_{ij},$ since then one takes $\oset' = \bigcap_{i,j} \oset _{ij}.$

The idea underlying the definition is that there is a dense subset
of $\oset$, at  each point of which  only one of the $f_i$
dominates.
\end{rem}

\subsection{Polynomials}

 Any polynomial can be viewed naturally
in $\CFunR$. We say that two polynomials are \textbf{e-equivalent}
if their images in $\CFunR$ are the same; i.e., if they yield the
same function from $R^{(n)}$ to $R$. Abusing notation, we
sometimes write $f(\la_1, \dots, \la _n)$ for a polynomial $f \in
R[\la_1, \dots, \la _n]$, indicating that $f$ involves the
variables $\la_1, \dots, \la _n$.

We say that $f_j$ is \textbf{essential} in $f = \sum _i f_i \in
\CFunR$ if there exists some nonempty open set $\oset ' \subset R
^{(n)}$ for which
 $$f_j(\bfa)^\nu
> \sum _{i\ne j} f_i(\bfa)^\nu \quad  \ \text{for all  } \bfa =(a_1,\dots,a_n) \in
\oset '.$$  We define the set $\oset _{f_j}$ to be the set
  $\{ \bfa \in R^{(n)} : f_j (\bfa) \text { dominates } (\bfa)\}$.

  The case of an essential
monomial of a polynomial, defined in
\cite{IzhakianRowen2007SuperTropical}, is a special case of this
definition.  The \textbf{essential} part of a polynomial $f$ is
the sum of its essential monomials. Since the essential part of
$f$ has the same image in $\CFunR$ as $f$, we may assume that the
polynomials we examine are essential. Note that a polynomial is
ghost (as in Definition~\ref{ghost0}) iff its essential part is a
sum of ghost monomials.

\begin{rem} By definition, for any tangible function, there is
a nonempty open set on which it cannot be ghost. Thus, a tangible
essential summand of a polynomial $f$ must dominate at some
tangible value.
\end{rem}

Recall that the point $\bfa = (a_1, \dots, a_n) \in R^{(n)}$ is
called a \textbf{root} of a polynomial $f(\la_1, \dots, \la_n)$
iff $f(\bfa)$ is ghost. For $n=1,$ we say that a root of $f(\la)$
is \textbf{ordinary} if it is a member of an open interval that
does not contain any other roots of $f$. Likewise, a common root
of two polynomials $f(\la)$ and $g(\la)$ is \textbf{2-ordinary} if
it is a member of an open interval that does not contain any other
common roots of $f$ and ~$g$. (More generally, for $n > 1$, a root
$\bfa \in R^{(n)}$  of $f$  is said to be \textbf{ordinary} if
$\bfa$ belongs to some  open set $\oset_\bfa $ which contains a
dense subset $W'$ on which $f$ is tangible, but we only consider
the case $n=1$ in this paper.)

\begin{lem}\label{helpgh} Suppose $f = \sum _i f_i \in \CFunR$ is ghost on some nonempty open
set $\oset $ on which the $f_i$ are $\nu$-distinct. Then each
 summand $f_j$ of $f$ that is essential on $\oset $ is ghost on an open subset~$\oset_j$  of~$\oset $.\end{lem}
\begin{proof} Otherwise the subset of $\oset $ on which $f_j$ dominates
contains a tangible element, and thus contains a tangible open
set, contrary to hypothesis.\end{proof}

\begin{rem}
 We say that  a function $f$ is \textbf{tangible at} $\bfa\in R^{(n)}$ if $\bfa$
 is not a root of $f$, i.e., if $f(\bfa)\in \tT.$  We denote the set of these point as:
 $$\tT_f =\{ \bfa \in R^{(n)} \ : \ f(\bfa) \in \tT \}.$$

 Confusion could arise because a tangible polynomial need
not be tangible at every point. For example, the tangible
polynomial $(\la+2)(\la+1)$ is tangible  at all tangible points
\emph{except} at $2$ and $1$, where its  values are ghosts. It is
easy to see that a polynomial in one indeterminate is tangible iff
it is tangible  at all but a  finite number of the tangible
points. Thus, any polynomial whose essential coefficients are all
tangible is tangible.

 Given a
polynomial $f(\la_1, \dots, \la _n) = \sum \a_\bfi
\la_1^{i_1}\cdots \la _n ^{i_n},$   we define $\hat f$ to be $\sum
\hat \a_\bfi \la_1^{i_1}\cdots \la _n ^{i_n},$ a tangible
polynomial according to Definition~\ref{ghost0} (although $\hat
f(a)$ need not be tangible for $a\in \tT$). Note that $\hat f
(a)^\nu = f(a)^\nu$.
 \eroman
\end{rem}

Recall that the \textbf{supertropical determinant} $\Det{A}$ of a
matrix $A = (a_{ij})$ is defined to be the permanent, i.e.,
$\Det{A}= \sum_{\sig \in S_n} a_{1, \sig(1)} \cdots
a_{n,\sig(n)}$; cf.~\cite{IzhakianRowen2008Matrices}.

\section{Transformations of supertropical varieties}

 The \textbf{root set} of $f \in R[\lm_1, \dots, \lm_n ]$ is the
set
 $$Z(f) = \{  \bfa \in R^{(n)} \ | \ f(\bfa) \in \tGz\},$$
and $Z_{\tan}(f)= Z(f)\cap \tTz^{(n)} $ is called the
\textbf{tangible root set} of $f$.

The tangible root set provides a tropical version of affine
geometry; analogously, one would define the supertropical version
of projective geometry by considering equivalence classes of
tangible roots of homogeneous polynomials (where, as usual, two
roots are projectively equivalent if one is a scalar multiple of
the other). There is the usual way of viewing a polynomial
$f(\la_1, \dots, \la_{n})$ of degree $t$ as the homogeneous
polynomial $\la_{n+1}^t f(\frac{\la_1}{\la_{n+1}},
\dots,\frac{\la_n}{\la_{n+1}}) $, and visa versa. Since the
algebra is easier to notate in the affine case, we focus on that.

We need to be able to find transformations of supertropical root
sets, in order to move them away from ``bad'' points.

\begin{rem}\label{transf1} Suppose $f(\la_1, \dots, \la_n)\in
R[\lm_1, \dots, \lm_n].$

\begin{enumerate} \eroman
    \item Given   $\bfb = (\bt_1, \dots, \bt_n) \in \tT^{(n)}$, we define the
\textbf{multiplicative translation}
$$f_{(\bfb,\cdot)} = f(\beta_1\la_1, \dots, \beta_n\la_n).$$
Clearly, when the $\beta_i$ are invertible,
$$Z_{\tan}(f_{(\bfb,\cdot)}) = \{ (\bt_1^{-1}a_1 , \dots,
\bt_n^{-1}a_n): (a_1  , \dots, a_n) \in Z_{\tan}(f) \}.$$ Thus,
the roots of $f_{(\bfb,\cdot)}$ are multiplicatively translated by
$b$ from those of $f$. \pSkip

    \item  Given   $\beta   \in R$, define the
\textbf{additive translation}
$$
f_{(k,\bt,+)} = f(\la_1  , \dots, \la_{k-1}, \la_k +\beta,
\la_{k+1}, \dots, \la_n).$$ If $\bfa \in Z_{\tan}(f) $ where $\bfa
= ( a_1, \dots, a_n)$ with $a_k^\nu$ ``sufficiently small,'' then
$\bfa \in Z_{\tan}(f_{(k,\bt,+)}).$ Indeed, writing $f = \sum f_j
\la_k^j$ where $\la_k$ does not appear in $f_j,$ and dividing
through by the maximal possible power of $\la_k$, we may assume
that $f_0$ is nonzero, and thus dominates any root $\bfa = (a_1,
\dots, a_n)$ whose $k$-th component has small enough $\nu$-value.
Hence $f_0(a_1, \dots, a_{k-1}, a_{k+1}, \dots, a_n)$ is ghost,
and this dominates in $f_{(k,\bt,+)}$ as well as in $f$.
\end{enumerate}

\end{rem}

\subsection{The partial Frobenius morphism}

Another transformation comes from a morphism of supertropical root
sets which arises from  the \textbf{Frobenius property}, which we
recall from \cite[Remark
3.22]{IzhakianRowen2007SuperTropical}: \\
There is a semiring endomorphism $$ \frb: \Fun(R^{(n)},R )\To
\Fun(R^{(n)},R )$$  given by $\frb: f \mapsto f^m.$ We want to
refine this  for polynomials.

\begin{defn} Define the $k$-th $m$-Frobenius map ${\frb_k}^m
: R[\lm_1, \dots, \lm_n] \to R[\lm_1, \dots, \lm_n]$ given by
$${\frb_k}^m : f(\la_1, \dots, \la_n) \longmapsto  f(\la_1, \dots, \la_{k-1}, {\la_k}^m, \la_{k+1},\dots,
\la_n).$$
\end{defn}

\begin{lem}\label{Frobmor0} For any $k$ and $m$, the  $k$-th $m$-Frobenius map ${\frb _k}^m$ is a homomorphism of
semirings, which is in fact an automorphism when $m\in \Net$ and
  $R$ is a divisibly closed supertropical semifield.
\end{lem}
\begin{proof} Writing $f = \sum _{\bfi} \al _\bfi \la _1^{i_1}
\cdots \la _n^{i_n},$ summed over $\bfi  = (i_1, \dots, i_n)
\subset \Net^{(n)},$ we have $${\frb _k}^m(f) = \sum \al _\bfi \la
_1^{i_1} \cdots \la _k^{m i_k}\cdots \la _n^{i_n}.$$ It follows
just as in \cite[Proposition~3.21]{IzhakianRowen2007SuperTropical}
that for $g = \sum \bt _\bfi \la _1^{i_1} \cdots \la _n^{i_n},$
$${\frb _k}^m(f+g) = {\frb _k}^m(f)+{\frb _k}^m(g),$$ and
clearly
$${\frb _k}^m(f\, g) =
{\frb _k}^m(f)\; {\frb _k}^m(g).$$
\end{proof}

\begin{rem}\label{Frobmor} In the set-up of Lemma \ref{Frobmor0},   each Frobenius map ${\frb _k}^m$ defines a morphism of
root sets, given by $$(a_1, \dots, a_n) \mapsto (a_1, \dots,
a_{k-1}, {a_k}^{\frac 1 m}, a_{k+1}, \dots, a_n),$$ which we call
the \textbf{partial Frobenius morphisms}.
\end{rem}

\subsection{Supertropical Zariski topology and the generic method}

Since one of the most basic tools in algebraic geometry is Zariski
density, we would like to utilize the analogous tool here:
\begin{rem}\label{Zarden} Any
polynomial formula expressing equality of $\nu$-values that holds
on a dense subset of~$R^{(n)}$ must hold for all of $R^{(n)},$
since polynomials are continuous functions.
\end{rem}

Such a density argument is   used in Section~\ref{resulta}. There
is an alternate method to Zariski density for verifying that
identical relations holding for tangible polynomials must hold for
arbitrary polynomials. It is not difficult to write down a generic
polynomial over a semiring with tangibles and ghosts. Namely, we
let $\tilde R = R[\mu _0, \dots, \mu _t]$ where the $\mu_i$ are
indeterminates over $R$, and view the polynomial $\sum _{i=0}^t
\mu_i \la ^i \in \tilde R[\la];$ any polynomial $f = \sum \a_i \la
^ i \in R[\la]$ can be obtained by specializing the $\mu_i$
accordingly. However, one has to contend with the following
difficulty: Although this new semiring with ghosts $\tilde R$
satisfies supertropicality, it is not a supertropical semiring,
and so
 identical relations holding in supertropical semirings may well fail in $\tilde R$.

\section{Supertropical polynomials in one indeterminate}

This section is a direct continuation of
\cite{IzhakianRowen2007SuperTropical}; we focus on properties of
 common tangible roots of polynomials in the supertropical setting. Assume
throughout this section that $F$ is an $\Net$-divisible
supertropical semifield, with ghost ideal $\tGz$ and tangible
elements $\tT$. We view polynomials in $F[\lm]$  as functions,
according to their equivalence classes in $\CFunFF1 $, or
equivalently we consider the full polynomials {\cite[Definition
6.1]{IzhakianRowen2007SuperTropical}} which are their natural
representatives. Thus,
$$
\begin{array}{rcl}
\text{a polynomial $f(\la)$ is ghost} &  \iff   &  f(a)
\text{ is ghost for each } a \in F; \\[1mm]
 \text{a polynomial $f(\la)$ is
tangible}  &  \iff   & f(\oset_\tT ) \text{ is not ghost for each
tangible open interval } \oset_\tT \subset F .
\end{array}
$$

 A polynomial is called
\textbf{monic} if its leading coefficient is $\one_F$ or
$\one_F^\nu$ (i.e., $0$ or $0^\nu$ in logarithmic notation).

 We recall the following factorization:
\begin{thm}\label{factorpol}{\cite[Theorem 7.43 and Corollary 7.44]{IzhakianRowen2007SuperTropical}} \label{fullfact2}
Any monic full polynomial in one indeterminate has a unique
factorization of the form $f = \ft \fm$, where the
\textbf{tangible component} $\ft$ is the maximal product of
tangible linear factors $\lm + a_i$, and the \textbf{intangible
component} $\fm$ is a product of irreducible quadratic factors of
the form $\lm^2 +{ b_j}^\nu\lm + c_j$, at most one linear left
ghost $\lm^\nu + a_\ell$ and at most one linear right ghost $\lm +
a_r ^\nu$. (One obtains $\ft$ and $\fm$ from the factorization
called  ``minimal in ghosts.'')
\end{thm}

\begin{remark}\label{rmk:irr}
The factors of $\fm$ as described in the theorem are all
irreducible polynomials in $F[\lm]$. Furthermore, by \cite[Theorem
7.43]{IzhakianRowen2007SuperTropical}, their sets of tangible
roots are disjoint, and in fact one can read off these irreducible
quadratic factors from the connected components of
$Z_{\operatorname{tan}}(f)$.
\end{remark}

Denoting the linear tangible terms  as $p_i = \lm +a_i$ and the
quadratic terms as $q_j= \lm^2 + b_j^\nu\lm + c_j$, we write
\begin{equation}\label{eq:fact} f = (\lm^\nu + \al_\ell) (\lm +
\al_r^\nu) \prod_i p_i \prod_j q_j
\end{equation}
for this factorization of $f$ which is minimal in ghosts.

We say that a polynomial $g(\la_1, \dots, \la _n)$
\textbf{e-divides} $f(\la_1, \dots, \la _n)$ if, for a suitable
polynomial $h$, the polynomials  $f$ and $gh$ are e-equivalent. (A
weaker concept is given below, in Definition \ref{superdiv}).

\begin{remark}\label{rmk:tangFact} \begin{enumerate}
    \item Any tangible  polynomial of degree $n$ has at most $n$ distinct
tangible roots. \pSkip \item  If $f \in F[\lm]$ is a tangible
polynomial of degree $n$, then $f$ e-factors uniquely into $n$
tangible linear factors.

\end{enumerate}

\end{remark}

\begin{prop}\label{tangdiv} Suppose a polynomial $p$ e-divides $fg$ for $g$ tangible, and
$p$ is  irreducible nontangible. Then $p$ e-divides $f$.
\end{prop}
\begin{proof} Write $f = \ft \fm$ as in
\eqref{fullfact2}
 where $\ft$ is the tangible component of $f$. Then $fg = \ft \fm g $ and
 $\ft g$ is tangible; hence $\ft g= (fg)^\tng$ and $p$ must divide $(fg)^{\ntng} = \fm$.
\end{proof}

We turn to the question of how to compare polynomials in terms of
their roots. The next example comes as a bit of a surprise.

\begin{example}\label{lem:sumPoly0}
Some examples of polynomials $f,g\in F[\la]$ such that $f+g$ is
ghost, but $f$ and $g$ have no common tangible root.
\begin{enumerate} \eroman
    \item  Suppose $f$ is a full polynomial, all but one of whose
monomials $h$ have a ghost coefficient, and $g=h$. For example,
take $f = (\la^2)^\nu + 2 \la + 3^\nu$ and $g = 2\la.$ Then $f+g$
is obviously ghost, but $g$ has no tangible roots at all; thus,
$f$ and $g$ have no common tangible roots. \pSkip

    \item

 In logarithmic notation, where $F = (\Real,\max,+)$, take $f
= \la (\la^\nu +1) = (\la^2)^\nu +1\la,$ and $g =  1\la + 0^\nu.$
Then
\begin{equation} f(a) =
\left\{%
\begin{array}{ll}
    1a  & \hbox{if } \ a^\nu <
1^\nu; \\
    (a^2)^\nu & \hbox{if } \ a^\nu \ge 1^\nu. \\
\end{array}%
\right.
\end{equation}
In particular, $f(a)$ is tangible for all $a$ on the tangible open
interval $(-\infty, 1). $ Also,
\begin{equation} g(a) =
\left\{%
\begin{array}{ll}
    0^\nu   & \hbox{if }  \ a^\nu \le
{-1}^\nu; \\
     1a   & \hbox{if }  \ a^\nu >  -1^\nu. \\
\end{array}%
\right.
\end{equation}
In particular, $g(a)$ is tangible for all $a$ on the tangible open
interval $(-1, \infty).$ Thus $f$ and $g$ have no common tangible
roots, although $f+g$ is ghost (since $f(a)=g(a)$ for all $a \in
(-1,1)).$

\end{enumerate}
\end{example}

One can complicate this example, say by taking $f = (\la^3)^\nu+
3(\la^2)^\nu +3 \la = (\la +3)(\la^\nu +0)\la$ and $g = 3\la^2+
2\la^\nu +0^\nu.$ Nevertheless, these are the ``only'' kind of
counterexamples, in the   sense of the
Proposition~\ref{lem:sumPoly} below.

\subsection{Graphs and roots}

In this subsection, we assume that the supertropical semifield $F$
is connected, in order to apply some topological arguments.

\begin{defn}\label{graphs} The \textbf{graph}  $\gr_f$ of a  function  $f\in \CFunF $ is defined as
the set of ordered $(n+1)$-tuples $(\bfa,f(\bfa))$ in $F^{(n+1)}$,
where $\bfa = (a_1, \dots, a_n) \in F^{(n)}$. Note that either
component of $\gr_f$ could be tangible or ghost, so in a sense the
graph has at most $2^{n+1}$ leaves.

The $\tG$-\textbf{graph} $\grg_f$ is defined as $\{ (\bfa^\nu,
f(\bfa)^\nu): \bfa \in F^{(n)} \}$; i.e., we project onto the
ghost values. (Note that if $\bfa^\nu = \bfb^\nu,$ then
$f(\bfa)^\nu = f(\bfb)^\nu$.) It is more convenient to consider
the \textbf{tangible} $\tG$-\textbf{graph} $$\grgTf{f} = \{
(\bfa^\nu, f(\bfa)^\nu): \bfa \in \tT^{(n)} \};$$ $\grgTf{f}$ can
be drawn in $n+1$ dimensions.
\end{defn}

In this paper, we   consider a polynomial $f \in F[\la]$ in one
indeterminate, so its $\tG$-graph lies on a plane, and
 is a sequence of line segments which can change slopes only at
 the tangible
 roots of $f$. We can describe the essential and quasi-essential monomials of $f$ as in
 ~\cite{IzhakianRowen2007SuperTropical}: Writing $f = \sum \al_i \la^i,$ and defining the slopes
 $\gamma _i = \frac
{\hat \al_{i+1}}{\hat \al_{i}},$  we see that the monomial $h =
\al _i \la ^i$ is essential only if $\gamma_{i-1}^\nu < \gamma
_i^\nu $, and $h $ is quasi-essential only if $\gamma_{i-1}^\nu =
\gamma _i^\nu  $. Note that when the monomial $h $ is essential
(at a point $a$), the $\tG$-graph $\grg_f$ for $f$ must change
slope at $a$. We say that a polynomial is \textbf{full} if each of
its monomials is essential or quasi-essential.

\begin{defn}\label{lefthalf}
We say that $f \in F[\la]$ is $\a$-\textbf{right} (resp.~
\textbf{left})  \textbf{half-tangible} for $\a \in \tG$ if $f$
satisfies the following condition for each $a\in \tT$:  $$f(a) \in
\tT \quad \text{ iff }\quad a^\nu
> \a
  \quad \text{(resp.}\quad  a^\nu < \a  ),$$
  which implies $f(a) \in \tGz$ for all $a \in R$ with $a^\nu < \al$ (resp.  $a^\nu > \al$).
  \end{defn}
 By definition, if $f$ is $\a$-right half-tangible, all roots
 of $f$ must have $\nu$-value $\le \a,$ and thus the tangible $\tG$-graph~$\grgTf{f}$ of $f$
 must have a
 single ray emerging from  $\a$. (The analogous assertion holds for left
 half-tangible.)

\begin{proposition}\label{lem:sumPoly}
If $f,g\in F[\la]$ are polynomials without a common tangible root,
with neither $f$ nor~$g$ being monomials,  and $f+g$ is ghost,
then $f$ is left half-tangible and $g$ is right half-tangible (or
visa versa); explicitly, there are $\a < \bt$ in $\tG$  such that
$f$ is $\bt$-left half-tangible, $g$ is $\a$-right half-tangible,
and $f(a)^\nu = g(a)^\nu$ for all $a$ in the tangible interval
$(\a, \bt)$. Furthermore, in this case, $\deg (f) > \deg (g)$ (and
likewise the degree of the lowest order monomial of $g$ is less
than the degree of the lowest order monomial of $f$.)
\end{proposition}
\begin{proof} In order for $f+g$ to be ghost,  $(f+g)(a)$ must be ghost for each $a \in F$,
 which means that either:
\begin{enumerate}
\item $f(a)$ is ghost of $\nu$-value greater than $g(a)$, \pSkip
\item $g(a)$ is ghost of $\nu$-value greater than $f(a)$, or \pSkip
 \item $f(a)^{\nu} = g(a)^{\nu}$.\end{enumerate}
  Let $\oset  _{f; \tT} $ (resp.~ $\oset  _{g ; \tT}$) denote the  (open) set of tangible elements satisfying Condition (1)
  (resp.~(2)). We are done unless $\oset  _{f; \tT} $ and $\oset  _{g; \tT}$ are
  disjoint,
  since any element of the intersection would be a common tangible root of $f$ and $g$.

Note that $f(a)$ must be ghost for every element $a$ in the
closure of $\oset _{f; \tT}$. (Indeed, if $f(a)$ were tangible
there would be some tangible interval $\osetb_\tT$ containing $a$
for which all values of $f$ remain tangible; then, $\osetb_\tT
\cap \oset _{f; \tT} \neq \emptyset$, contrary to definition
of~$\oset  _{f; \tT}$.) Likewise, $g(a)$ is ghost for every
element $a$ in the closure of~$\oset_{g; \tT}$.

If  $\oset  _{f; \tT} = \emptyset$, then $Z_{\tan}(g) = \tT$, and
any tangible root of $f$ is automatically a root of $g$. Hence, we
may assume that $\oset _{f; \tT}$ and likewise $\oset _{g; \tT}$
are nonempty.

Also, let $$S_\tT = \{ a\in \tT : f(a)^{\nu} = g(a)^{\nu}\}.$$ Let
$S_{f; \tT} = \{a \in S_\tT: f(a) \text{ is tangible} \}$ and
$S_{g; \tT} = \{a \in S: g(a) \text{ is tangible} \}$.  Since any
$a \in \tT$ cannot be a common root of $f$ and $g$, we must have
$f(a)$ or $g(a)$ tangible, thereby implying $S_{f; \tT} \cup S_{g;
\tT}= S_\tT.$ As noted above, $S_{f; \tT}$ is disjoint from the
closure of $\oset _{f; \tT} $.

Suppose $a$ is a tangible element in the boundary of~$\oset  _{f;
\tT} $ (which by definition is the complement of $\oset _{f; \tT}
$ in its closure). Then $f(a)^{\nu} = g(a)^{\nu}.$ As noted above,
$f(a)$ must be ghost; if $a$ also lies in the closure of $\oset
_{g; \tT} $, then $g(a)$ is also ghost, contrary to the hypothesis
that $f$ and $g$ have no common tangible roots. Since $\tT$ is
presumed connected, we must have $S_{f; \tT}\cap S_{g; \tT} \ne
\emptyset.$

Write $S_{f; \tT} \cap S_{g; \tT}$ as a union of disjoint
intervals, one of which we denote as $(\a, \bt)$. For $a'$ of
$\nu$-value slightly more than $\bt,$ suppose $a' \in \oset _{f;
\tT}.$ Then the slope of the tangible $\tG$-graph $\grgTf{f}$ of
$f$ at $a'$ must be at least as large as the slope of $\grgTf{g}$
at $a'$, and this situation continues unless $g$ has some tangible
root $a \in \oset _{f; \tT},$ contrary to hypothesis. Thus,
$g(a)^\nu < f(a)^\nu$ for each $a$ of $\nu$-value $> \beta,$
implying $f(a) \in \tGz$ for all such $a$, and thus, by
hypothesis, $g(a)\in \tT$ for all  $a$ of $\nu$-value $> \beta$.

We have also proved that $S_{f; \tT} \cap S_{g; \tT} = (\a, \bt)$
is connected, and its closure is all of $S_\tT$ since otherwise
$S_\tT$ has a tangible point at which the $\tG$-graphs,
$\grgTf{f}$ and $\grgTf{g}$, both change slopes and thus must both
have a tangible root. Hence, $f$ and $g$ are both tangible on the
interior of $S_\tT$.

By hypothesis, $g$ is not a monomial, and thus has some tangible
root, which must have $\nu$-value $< \a$. The previous argument
applied in the other direction (for small $\nu$-values) shows that
$g(a)$ is ghost and $f(a)$ is tangible for all $a$ of $\nu$-value
$< \a$.

Finally, since $f$ increases faster than $g$ for $a^\nu
>\bt,$ it follows at once that $\deg( f) > \deg( g);$  the last assertion follows
by symmetry. \end{proof}

Conversely, if $f,g$ satisfy the conclusion of
Proposition~\ref{lem:sumPoly}, then clearly $f+g$ are ghost. Thus,
a pair of polynomials whose sum is ghost is characterized either
as having a common tangible root or else satisfying the conclusion
of Proposition~\ref{lem:sumPoly}. In particular, two polynomials
of the same degree whose sum is ghost must have a common tangible
root.

\begin{example}\label{lem:sumPoly0} It is also instructive to consider the following example:
  $$f = (\lm+ 2)(\lm + 5 ^\nu)
  (\la +8^\nu)(\la +9), \ \text{ and  } \ g =  (\la +3)(\la +4)(\lm^\nu +7)(\lm + 10)
  .$$
We have the following table of values for $f$ and $g$:
$$\begin{array}{l|llllllllllll} a & 2 & 3 & 4 & 5 & 6 & 7 & 8 & 9 & 10 & 11 &
\dots
\\  \hline f(a) & 24^\nu & 25^\nu & 26^\nu & 27^\nu & 29^\nu & 31^\nu &   33^\nu  &
36^\nu &   40 &   44 & \dots
\\ g(a) & 24 & 24^\nu & 25^\nu & 27& 29  & 31^\nu &  34^\nu  &
37^\nu &   40^\nu &   44^\nu & \dots
\end{array}$$

Note that $\deg (f) = \deg (g) = 4$ and $f+g = \nu( \lm^4 + 10
\lm^3 +17 \lm^2 + 22 \lm +  24 )$ is ghost, whereas they have
exactly three ordinary common tangible roots, namely $3,4,$ and
$9$; each $a \in [7,8]$ is also a common tangible root.
\end{example}

\subsection{Relatively prime polynomials}

 In order to compare polynomials in terms of their roots, we need another notion. Given a polynomial
$f$, we write $\ldeg(f)$ for the degree of the lowest order
monomial of $f.$ For example, $\ldeg(\la^3 + 2 \la^2 + \la ^\nu )
= 1.$

\begin{defn}
Two   polynomials $f$ and $g$ of respective degrees $m$ and $n$
are \textbf{relatively prime} if there do not exist tangible
polynomials $ \tp$ and $ \tq$ (not both $\fzero$) with $\deg( \tp
)< n$ and $\deg ( \tq )< m,$ such that $ \tp f+  \tq g$ is ghost
with $\deg( \tp f)=\deg( \tq g)$ and $\ldeg( \tp f)=\ldeg( \tq
g)$.
\end{defn}

We say that two polynomials $f$ and $g$ have a \textbf{common
$\nu$-factor} $h$ if there are polynomials $h_1,h_2$ with $h_1^\nu
= h_2^\nu= h^\nu$,  such that $h_1$ e-divides $f$ and $h_2$
e-divides $g.$

\begin{remark}\label{relpr2}
Any monic polynomials $f$ and $g$ having a common $\nu$-factor $h$
are not relatively prime. Indeed, write $f= h_1 q$ and $g = h_2 p$
and thus
$$\tp f  + \tq  g =\tp h_1 q  + \tq  h_2 p = h_1   \tp q +  h_2 p \tq  =
\ghost$$ (since they are $\nu$-matched). On the other hand, two
non-relatively prime polynomials without a   common factor could
be irreducible; for example for $\lm + 2^\nu$ and $\lm +1$ we have
$(\lm + 2^\nu) 1 +(\lm +1) 1$ is ghost, but both are irreducible,
cf. Remark \ref{rmk:irr}.
\end{remark}

\begin{remark} \label{relpr1}  $ $
\begin{enumerate}
    \item
 If $\ldeg (f)$ and  $\ldeg (g)$ are both positive, then $f$ and $g$ cannot
be relatively prime, since they have the common $\nu$-factor
$\la$. Similarly, if $\ldeg (f) = 0$ and $\ldeg (g) > 0,$ then one
can cancel $\la$ from $g$ without affecting whether $g$ is
relatively prime to $f$. Thus, the issue of being relatively prime
can be reduced to polynomials having nontrivial constant term. But
then, cancelling powers of $\la$ from $ \tp$ and $\tq ,$ we may
assume that $ \tp$ and $ \tq $ also have nontrivial constant term.
Thus, $\ldeg( \tp f)=\ldeg( \tq g) = 0,$ so the condition that
their lower degrees match is automatic. \pSkip

\item Adjusting the leading coefficients in the definition, we may
assume that $f$ and $g$  are both monic. (However, $\tp$ and $\tq$
need not be monic, as evidenced taking $f = \la ^\nu +1$ and $g =
\la +3$ in logarithmic notation; then $2f+g$ is ghost.) \pSkip

\item  A nonconstant ghost polynomial $f$ cannot be relatively
prime to any nonconstant polynomial~$g$, since   $fh + g$ or
$f+hg$ is ghost, where $h$ is any polynomial of degree $|\deg
f-\deg g|$ with ``large enough'' coefficients or ``small enough''
coefficients respectively. \pSkip

\item   If $ \tp f + \hat  q g  $ is  ghost with $ \tp = (\la +a)
\tp_1$ and $\tq = (\la +a)\tq_1$ then $(\la +a)( \tp_1 f + \tq_1 g
)$  is  ghost. Hence, $ \tp_1 f + \tq_1 g $ is ghost at every
point except $a$, which implies $ \tp_1 f + \tq_1 g $ is ghost, by
continuity.
\end{enumerate}
\end{remark}

We also need the following observation to ease our computations.

\begin{lemma}\label{ghostmatch} Suppose the polynomial $f+g$ is ghost,
and $p,q\in R[\la]$ with $p^\nu = q^\nu.$ Then $pf +qg$ is also ghost.\end{lemma}
\begin{proof}  Write $p = \sum \al_i \la ^i$ and $q = \sum \bt _i \la^i$ where
 $\al_i^\nu = \bt_i^\nu.$ For any monomial $f_\ell$ of $f$ there is some monomial $g_k$ of
 $g$ such that $f_\ell +g_k$ is ghost. But any monomial of $pf$ has the form $\al_i \la^i f_\ell,$
  which when added to $\bt_i \la^i g_k$ is
clearly ghost. \end{proof}


\begin{theorem}\label{thm:rootPrime} Over a connected supertropical semifield $F$, two  non-constant monic polynomials $f$ and~$g$ in $F[\lm]$ are not relatively prime iff $f$ and $g$  have a
common tangible root.
\end{theorem}
\begin{proof}  We may assume that $f$ and $g$ are both monic.
In view of Remark \ref{relpr1}, we may also assume that $f$ and
$g$ are non-ghost, and have nontrivial constant term.

$(\Rightarrow)$ Suppose $f $ and $g $  are not relatively prime;
i.e., $\tp f + \tq g$ is ghost for some tangible polynomials $\tp$
and $\tq$, with  $\deg(\tp f) = \deg(\tq g)$ and $\ldeg(\tp f) =
\ldeg(\tq g)$. Since $\ldeg( f) = \ldeg( g)= 0,$ we may cancel out
the same power of $\la$ from both    $\tp$ and $\tq,$ and thereby
assume that $\tp f$ and $\tq g$ each have nontrivial constant
term. We proceed as in the proof of Proposition~\ref{lem:sumPoly},
but with more specific attention to the tangible $\tG$-graphs
$\grgTf{\tp f}$ and $\grgTf{ \tq g}$, cf.~Definition~\ref{graphs}.
We assume that $f$ and $g$ have no common tangible root. In other
words, $f(a)\in \tG$ implies $g(a) \in \tT$, for any $a \in \tT$,
and likewise $g(a)\in \tG$ implies $f(a) \in \tT$. Also, we may
assume that $\tp$ and $\tq$ have no common tangible root, by
Remark~\ref{relpr1}(4).

Let $\oset  _{\tp f ; \tT} = \{ a \in \tT: \tp f(a)^\nu> \tq
g(a)^\nu \}$, and $\oset  _{\tq g; \tT} = \{ a \in \tT:\tq
g(a)^\nu>\tp f(a)^\nu \}.$ By hypothesis, $\tp f(a')$ is ghost for
all $a'\in \oset _{\tp f ; \tT}.$ But any $a'\in \oset _{\tp f ;
\tT }$ is contained in a tangible open interval $\osetb_\tT $ for
which $\tp$ is tangible on $ \osetb _\tT \setminus \{ a'\},$ so by
assumption, $f(a)\in \tG$ for all $a \in \osetb _\tT  \setminus \{
a'\},$ and thus $f(a)\in \tG$ for all $a \in \osetb _\tT .$ For
all $a\in \oset _{ \tp f ; \tT} $, it follows that $f(a)\in \tG$
and thus $g(a) \in \tT$ . Likewise, for all $b\in \oset _{ \tq g
;\tT}$, we have $g(b)\in \tG$ and $f(b) \in \tT$.

Note that as we increase the $\nu$-value of a point, the slope of
the graph $\grgTf{f}$ of a polynomial $f$ can only increase;
moreover, an increase of slope in the graph indicates the
corresponding increase of degree of the dominant monomial at that
point. We write $\dom_{\tp f}(a)$ (resp.~$\dom_{\tq g}(a))$ for
the maximal degree of a dominant monomial of $\tp f$ (resp.~$\tq
g)$ at $a \in F$.

Let $$S_\tT = \tT \setminus (\oset  _{ \tp f ; \tT} \cup \oset _{
\tq g ; \tT } ) = \{ a \in \tT: \tp f(a)^\nu =\tq g(a)^\nu\}.$$
Clearly, $\dom_{\tp f}(a) = \dom_{\tq g}(a)$ for every $a$ in the
interior of $S _ \tT,$ since the graphs $\grgTf{\tp f}$ and
$\grgTf{\tq g}$ must have the same slope there.

By symmetry, we may assume that $f(a_\sml)\in\tG$ for $a_\sml^\nu$
small. The objective of our proof is to show that as $a \in F $
increases, any change in the slope of $\grgTf{\tp f}$ arising from
an increase of degree of the essential monomial of $f$  is matched
by corresponding roots of $\tq$, and thus $\deg (\tq )= \deg (f)$
(and $\deg (\tp )= \deg (g)$) --- a contradiction.

We claim that the graphs   $\grgTf{\tp f}$ and $\grgTf{\tq g}$ do
not cross at any single tangible point  (i.e. without some
interval in $S_\tT$). Indeed, consider an arbitrary tangible point
$a\in S_\tT$ at which the graphs of $\grgTf{\tp f}$ and
$\grgTf{\tq g}$ would cross, starting say with $\grgTf{\tp f}$
above $\grgTf{\tq g}$ before $a$ and $\grgTf{\tq g}$ above
$\grgTf{\tp f}$ after $a$.
 At this
intersection point, $\tp f(a)$ and $\tq g(a)$ must both be ghost,
so $f(b)$ must be ghost for $b$ of $\nu$-value $< a^\nu$ whereas
$g(b)$ must be ghost for $b$ of $\nu$-value $> a^\nu$. But this
yields a common root for $f$ and $g$ unless $f$ switches from
ghost to tangible and $g$ switches from tangible to ghost, so $a$
would be a common root of $f$ and $g$, yielding a contradiction.

This proves that any point $a$ at which the graphs $\grgTf{\tp f}$
and $\grgTf{\tq g}$ meet must lie on the boundary of~$S_\tT$.
 Continuing along $S_\tT$, suppose that $f$ has a root $b$ in the interior of $S_\tT$. Then the
slope of $\grgTf{\tp f}$ increases by some number matching the
increase~$k$ of degree in the essential monomial of $f$ at~ $b$;
this must be matched by an equal increase in slope in~$\grgTf{\tq
g}.$ But $g$ cannot have a root here, since $f$ and $g$ have no
common tangible roots; hence $b$ is a root of~$\tq$ of
multiplicity~$k$. Thus, all roots of $f$ in the interior of
~$S_\tT$ are matched by roots of $\tq$.

Next let us consider what happens between two points on subsequent
tangible intervals of $S_\tT$. At any boundary point $a'$ of
$S_\tT$, for $a$ of slightly greater $\nu$-value than $a'$, we
have $a \in \oset _{ \tp f ; \tT} \cup \oset _{ \tq g ; \tT};$ say
$a\in \oset _{ \tp f; \tT}$. This means $\dom_{\tp f}(a')
> \dom _ { \tq g} (a')$. Clearly $a'$ is a root of $\tp f,$ and furthermore,
since $g$ is tangible in $\oset _{ \tp f ; \tT},$ any increase in
$\dom_{ \tq g}(a')$ occurs because of changes in the essential
monomial of $\tq$, i.e., from roots of $\tq$. Thus, when we enter
$S_\tT$ the next time, say at $a''$, we see that $\dom_{ \tp
f}(a'') - \dom _{ \tp f}(a')$ is the number of roots of $\tq$
needed to increase the slope of $\tq g$ accordingly. But when we
are within $\oset _{ \tq g; \tT},$ there cannot be any tangible
roots of $f$, and thus the essential monomial of~ $f$ does not
change. Continuing until we reach $a'',$ we see that the only
 increase in degree coming from change of the
dominant monomial of $f$ must occur in $\oset _{\tp f; \tT} \cup
S_\tT$ and are thus matched by roots from $\tq$.

Looking at the whole picture, we see that both graphs $\grgTf{\tp
f }$ and $\grgTf{ \tq g}$ have slope 0 for small $\nu$-values of
$a$ (since both $\tp f $ and $ \tq g$ have nontrivial constant
terms). Either they coincide for small $\nu$-values of~$a$, and we
start in $S_\tT$, or else one is above the other. Assume that
$\grgTf{\tq g}$ starts above $\grgTf{\tp f}$. But any increase of
slope of $\grgTf{\tp f}$ entails the same increase of slope of
$\grgTf{\tq g}$ (since otherwise we would have a crossing at a
single tangible point), and thus a corresponding increase in $\deg
(\tp)$, since any tangible root of $f$ (before the crossing) would
be a common root of $f$ and $g$, contrary to hypothesis. Then the
crossing brings us to $S_\tT$, and we continue the argument until
the last interval in $S_\tT$, and then when we leave, the
analogous argument at the end shows that any increase in the upper
graph leads to a corresponding increase in the tangible polynomial
($\tp $ or $\tq$) in the other graph.

Combining these different stages shows that $\deg (\tq) \le \deg(
f),$ which is what we were trying to prove.

(Symmetrically, any contribution to $\dom_{\tq g}$ coming from
changes in the essential monomial of $g$ happens in $\oset _{ \tq
g} \cup S_\tT,$ and thus is matched by roots of $\tp$.)

$(\Leftarrow)$  Our strategy is to e-factor $f$ and $g$ into
$e$-irreducible polynomials, all of which have degree $\le 2$.
Thus, we suppose first that $f$ and $g$ are $e$-irreducible
polynomials of respective degrees $m$ and~$n$ ($\leq 2$) having a
common tangible root $a$, and consider the following cases
according to Theorem~\ref{fullfact2}: \pSkip

\emph{{Case I:}} Suppose $m = n =1$. If $f$ and $g$ are  both
tangible we are done,
    since then $f =g = \lm + a$. The cases when both $f$ and $g$ are linear left ghost
    or linear right ghost are also clear. Finally,   when $f = \lm^\nu  + \al_f$ and $g = \lm +
    \al_g^\nu$,  for $\al_f, \al_g \in \tT$, we must have $\al_f^ \nu  \leq   a ^ \nu  \leq \al_g^ \nu$. Thus $f+g = \lm ^\nu +
    \al_g^\nu$. \pSkip

\emph{{Case II:}} Suppose  $m = 2$ and $n = 1$, and let $f = \lm^2
+ \bt^\nu_f \lm +
    \al_f$ with $(\bt^2_f)^{\nu} >   \al_f^\nu$.
     For $g = \lm +
    \al_g^\nu$, we have $(\frac{\al_f}{\bt_f})^\nu  \leq a^\nu \leq \min\{\bt_f ^\nu,
    \al_g^\nu \}$, so $\al_f ^\nu \le (\bt_f \al_g  )^\nu$ and
    $$f + (\lm+ \bt_f)g =  (\lm^2)^\nu + (\bt_f^\nu  +\bt_f+\al_g^\nu )\lm + \al_f +(\bt_f \al_g )^\nu = (\lm^2)^\nu + (\bt_f^\nu +\al_g^\nu  )\lm   +(\al_g \bt_f)^\nu$$ is ghost.

     When $g = \lm^\nu +
    \al_g ,$ then $\max\left\{ \al_g^\nu, \big(\frac{\al_f}{\bt_f}\big) ^\nu \right\}  \leq a^\nu \leq
    \bt_f^\nu ,$ so $$f +  \(\lm + \frac{\al_f}{\al_g}\)g
    = \(\lm^2\)^\nu + \(\bt_f^\nu +\frac{\al_f}{\al_g}^\nu\) \lm +
    \al_f^\nu \ - \ \text{a ghost }
    .$$
    If $g = \lm +   \al_g $, then $ a = \al _g$ and  $\( \frac{\al_f}{\bt_f}\)^\nu \leq a ^\nu \leq
    \bt_f ^\nu$, implying  $f + (\lm + \frac{\al_f}{\al _g})g = \(\lm^2\)^\nu + \bt_f^\nu\ \lm +
    \al_f^\nu $ is ghost.

\emph{{Case III:}} Suppose $m=n =2$, and let  $f = \lm^2 +
\bt^\nu_f \lm + \al_f$ and $g = \lm^2 + \bt_g^\nu\lm +
    \al_g$, for $\al_f, \al_g \in \tT$  with $(\bt _f ^2) ^\nu >   \al_f^\nu $
and $(\bt _g ^2) ^\nu >   \al_g^\nu $.
      Then \begin{equation}\label{eq:2}
\max \bigg\{\(\frac{\al_f}{\bt_f}\)^\nu, \ \( \frac{\al_g}
{\bt_g}\)^\nu \bigg\} \ \le  \
     a^\nu \ \le \  \min\{ \bt^\nu_f, \bt^\nu_g
    \}.
\end{equation}
  By symmetry, we may assume that $\al_f ^\nu \ge \al_g ^\nu .$ We claim that there
  are   elements $x,y \in \tTz$ such that, for
   $\tp = \lm + x $ and $\tq = \lm + y$, the polynomial    \begin{equation}\label{eq:1} \tp f + \tq g = (\lm^3)^\nu + (\bt^\nu_f + x +
    \bt^\nu_g + y )\lm^2 + (x \bt^\nu_f  + \al_f + y \bt^\nu_g + \al_g ) \lm +
    (x \al_f  + y \al_g) \end{equation}
    is ghost.

 Indeed, take $y = \hat \nu\left(\max \{ \bt_f^\nu, \bt_g^\nu\}\right)$ and
$x =  \frac{\al_g}{\al_f}y$.    The constant term in \eqref{eq:1}
is $(y \al_g + y \al_g)=  y \al_g^\nu.$ Likewise, the coefficient
of $\lm^2$ is ghost since $\bt_f^\nu+ \bt_g^\nu$ dominates $y$ and
$x$. Finally, the linear term is ghost since $\bt_g ^\nu \ge
\(\frac{\al_f}{\bt_f}\)^\nu$
 by the inequality~\eqref{eq:2}, implying
 $$y \bt_g ^\nu \ge   \bt_f^\nu \frac{\al_f^\nu }{\bt^\nu_f} = \al_f^\nu.$$
(The case for $f$ or $g$   ghost is trivial, by
Remark~\ref{relpr1}.)

 In general, suppose $f$ and $g $ are not necessarily irreducible,
 and have the common tangible root $a\in \tT$.
Consider the factorizations of $f= \prod_i f_i$ and $g = \prod_j
g_j$  into irreducible (linear and quadratic) polynomials. Thus,
$a$ is a common tangible root of some $f_i$ and $g_j$ of
respective degrees $m_i,n_j \le 2$, and,  by the first part of the
proof, $\tp_i f_i  +  \tq_j g_j $ is ghost for suitable tangible
polynomials $\tp_i$ and $\tq_j$ with $\deg( \tp_i )< n_i$ and
$\deg ( \tq_j )< m_j$ and $\deg ( \tp_i f_i) = \deg(\tq_j g_j) $
and $\ldeg( \tp_i f_i)=\ldeg( \tq_j g_j)$. Let
$$ r = \prod_{t \neq i} f_t; \qquad s = \prod_{u \neq j}g_u .$$
Taking
  $\tp =   \tp_i \hat s$ and $q =  \tq_j \hat r,$ we have  $
\tp f +   \tq g$   ghost.  Indeed, since  $\tp_i f_i  + \tq_j g_j
$ is ghost, and $\hat r s$ and $r \hat s$ have the same
$\nu$-value, write
$$\tp f +   \tq g  =   \tp_i \hat s f_i  \prod_{t \neq i}  f_t +
\tq_j \hat r  g_j \prod_{u \neq j}g_u  =\tp_i  f_i \hat s r+ \tq_j
g_j \hat r s $$ which is ghost
 by Lemma~\ref{ghostmatch}, and the degrees clearly
match. \end{proof}

The contrapositive of Theorem \ref{thm:rootPrime} gives us the
following analog of part of B\'{e}zout's theorem:

\begin{corollary}\label{sharper}
Over a connected supertropical semifield $F$, if $f$ and $g$ and
are two polynomials with no tangible roots in common, then they
are relatively prime.
\end{corollary}

One should compare this result with \cite[Theorem
8.5]{IzhakianRowen2007SuperTropical}, which says that for any
  polynomials $f$ and $g$ in $F[\lm]$ with no tangible roots in
common, there exist tangible polynomials $\tp$ and $\tq$ such
that, in logarithmic notation,
\begin{equation}\label{alt} \tp f + \tq g = 0 +\la \cdot
\ghost.\end{equation} But this last result is not as sharp as
Corollary~\ref{sharper}, since $f = \la +0$ and $g = 2^\nu \la $
have the tangible root $0$ in common, whereas they satisfy
\eqref{alt}, taking $\tp = \tq = 0$.

\begin{example} Consider the polynomials $f = 2^\nu\la^2 + 4 \la$
and $g = \la + 1^\nu,$ and $\tq = \la +4.$  Then $g \tq = \la^2 +
4\la +5^\nu, $ so $f+g \tq = 2^\nu\la^2 + 4^\nu \la +5^\nu $ is
ghost. On the other hand, $f = 2\la (\la^\nu +2)$ has the same
roots as $\la^\nu +2$, whose tangible roots (the interval
$[2,\infty)$) are disjoint from those of $g$ (the interval
$(-\infty, 1]$). Note that $\deg (f) = \deg(\tq g )$,  but $ 1=
\ldeg (f)
> \ldeg(\tq g )$.
\end{example}

\section{The resultant of two polynomials}\label{resulta}

In this section, we consider polynomials in one indeterminate over
an arbitrary supertropical semiring $R$, with an eye towards
applying induction and eventually dealing with polynomials in an
arbitrary number of indeterminates. Our main task is to determine
when two polynomials are relatively prime. This depends on the
essential parts of the polynomials. Since we want to use full
polynomials in the sense of \cite[Definition
7.10]{IzhakianRowen2007SuperTropical}, we assume   throughout this
section  that the polynomials $f$ and $g$ in $R[\la]$, of
respective degrees $m$ and $n$, are quasi-essential, in the sense
that every monomial is quasi-essential. In other words, writing $f
= \sum \a_i \la^i,$ we may assume that
$$ \(\frac{\hat \a_i}{\hat \a_{i+1}}\)^\nu \le
 \(\frac{\hat \a_{i+1}}{\hat \a_{i+2}}\)^\nu$$
 for each $i$. The classical
method for checking relative primeness of polynomials is via the
resultant, which has a natural supertropical version.

\begin{remark}\label{rmk:sylMat}
 For any semiring $R$, suppose $f = \sum _{i=0}^m \al _i \lm ^i \in
R[\lm]$, and let $A_n(f) $ denote the $n\times (m+n)$ matrix
$$\left(
\begin{array}{ccccccccccc}
 \al_0 & \al_1 & \al_2 & \al_3 & \dots & \al_m &    &
  & &   \dots &
\\
   & \al_0 & \al_1 & \al_2 & \dots & \al_{m-1} & \al_m &     &   &    &    \\
   &    &  \al_0 & \al_1 & \dots & \al_{m-2} & \al_{m-1} & \al_m &  &    &     \\
 \vdots &  &   & \ddots & \ddots & \ddots & \ddots
 & \ddots &   &  & \vdots \\
   &    &   &  & \al_0 & \dots  & \dots  & \dots  & \dots   & \al_m  &     \\
   &   \dots  &   &  &  & \al_0  & \al_1  & \dots  & \dots   & \al_{m-1} & \al_m \end{array} \right),$$
where the empty places are filled by $\rzero$. Then for any
polynomial $p= \sum _{i=0}^{n-1} \gm _i \la^i,$ we have
$$\begin{array}{cccc} ( \gm_0 & \gm_1 & \dots & \gm_{n-1}) \end{array} A _n(f)=
\begin{array}{cccc}  (\mu_0 & \mu_1 & \dots  & \mu
_{m+n-1})\end{array},$$ where $pf = \sum _{i=0}^{m+n-1} \mu _i\la
^i.$ (This is seen by inspection, just as in the classical
ring-theoretic case, since negatives are not used in the proof).
\end{remark}

\begin{definition}
The  \textbf{resultant matrix} $\res(f,g)$ is the $m+n$ square
matrix $$ \res(f,g) = \(
\begin{array}{c}
   A_n(f) \\
  A_m(g) \\
\end{array}
   \).$$ The \textbf{resultant} of $f$ and
$g$ is the  \permanent\   $\dtr{ \res (f,g) }.$ When $g = \bt$ is
constant, we formally define $\dtr{ \res (f,g) } = \beta^m.$
(Thus, when both $f$ and $g$ are constant, $\dtr{ \res (f,g) } =
\one_R.)$
\end{definition}

\begin{remark}$ $
\begin{enumerate} \eroman
    \item

 $|\res(f,g)|=  |\res(g,f)|,$ since we pass from  $\res(f,g)$ to $\res(g,f)$  by permuting rows. \pSkip

\item  If $f =  \sum _{i=0}^m \a_i \la ^i$ and $g =  \sum _{j=0}^n
\bt_j \la ^j$, then $$ |\res(f,g)| = \left|
\begin{array}{ccccccc}
             \al_0  & \al_1 & \al_2 & \dots & \al_m &  & \cdots \\
            & \al_0  & \al_1 & \al_2 & \dots & \al_m  &   \\
             \vdots & & \al_0  & \al_1 & \al_2 & \dots   &   \\
                &   &  &  \ddots  &\dots &\ddots &  \vdots \\
                     \bt_0 & \bt_1 & \bt_2 & \dots & \bt_n &  &  \\
                   &  \bt_0 & \bt_1 & \bt_2 & \dots & \bt_n &     \\
                    &   &  \bt_0 & \bt_1 & \bt_2 & \dots  & \vdots     \\
                       &&      &\ddots &\dots &
                  \end{array} \right|.$$
For any tangible $c$, dividing each of the last $m$ rows by $c$
 shows $|\res(f,g)|= c^m|\res(f,\frac 1{c}g)|.$ Thus, it is easy to reduce to
 the case that  $g$ is monic,
 and likewise for $f$. We often make this assumption without
 further ado. \end{enumerate}
\end{remark}

We need to compute the precise $\nu$-value of $|\res(f,g)|$.
Towards this end, the following remark is useful.

\begin{remark}\label{nuremove} $|\res(f,g)|^\nu =|\res(\hat f,\hat g)|^\nu$. Indeed, by
definition, the entries of the matrices whose determinants define
$|\res(f,g)|$ and $|\res(\hat f,\hat g)|$ have the same
$\nu$-values, so their determinants  have the same $\nu$-values.
\end{remark}

\begin{remark}\label{rmk:1}

 By Remark \ref{rmk:sylMat}, for any $p= \sum  _{i=0}^{n-1}\alpha _i
\la ^i$ and $q = \sum _{i=0}^{m-1}\beta _i \la ^i$ in $R[\la]$ of
respective degrees $n-1$ and $m-1$, with $pf+qg = \sum
_{i=0}^{m+n-1} \mu _i \la ^i, $  we have
\begin{equation}\label{eq:11}
\begin{array}{cccccc} (\alpha_0  & \dots & \alpha _{n-1} & \beta_0 & \dots & \beta _{m-1})
\end{array} \res(f,g) = \begin{array}{cccccc}  (\mu _0 & \mu _1
& \dots  & \mu _{m+n-1}).\end{array}
\end{equation}
\end{remark}

The direction of our inquiry is indicated  by the next
observation.
\begin{rem}\label{easydir} If $f,g \in F[\la]$ are not relatively prime, then
 $|\res(f,g)|$ is a ghost. (Just take tangible  $p,q$ of
respective degrees $\le  n-1$ and $m-1$ such that $pf+qg$ is a
ghost, and apply Remark \ref{rmk:1}.)\end{rem}

We look for the converse: That is, if $|\res(f,g)|$ is ghost, then
$f$ and $g$ are not relatively prime, and thus have a common
tangible root.

\begin{example}\label{specialcases2}  $ $
\begin{enumerate} \eroman
    \item
 Suppose $f = \sum_{i=0}^m \a_i \lm^i $ and  $g = \beta_1\lm  + \beta_0
$.  The resultant $|\res(f,g)|$ is given
 by:

$$ |\res(f,g)| = \left|  \begin{array}{ccccc}
             \a_0  & \a_1 & \a_2 & \dots & \a_m  \\
                     \beta_0 &\beta_1  & &  &   \\
                    & \beta_0 &\beta_1  & &   \\
                  &    &\ddots &\ddots & \\
                 &     &    & \beta_0 &\beta_1
                  \end{array} \right| =\a_0 \beta_1^m  + \a_1 \beta_0 \beta_1^{m-1} + \a_2 \beta_0^2 \beta_1^{m-2} +
                  \cdots + \a_m \beta_0^m.$$
In particular, if $\beta_1 = \one,$ then $|\res(f,g)|=
f(\beta_0),$ which is a ghost iff $\beta_0$ is a  root of $f$ (as
well as of $g$). We conclude for $\beta_0,\beta_1$ tangible that
$|\res(f,g)|=
 $  is a ghost iff $f$ and $g$ have a common root. (Indeed, first
 divide through by  $\beta_1$ to reduce to the case $\beta_1 = \one,$
 and then apply the previous sentence.)
\pSkip

\item  Suppose $f = \sum_{i=0}^m \a_i \lm^i $ and $g = \lm  +
b^\nu $, for $b$ tangible. As in (i),  the resultant $|\res(f,g)|$
equals $f(b^\nu)$, which is a ghost iff $b^\nu \ge (\frac
{\a_0}{\a_1})^\nu,$ the root of $f$ having smallest $\nu$-value.
Again, the resultant is a ghost iff $f$ and $g$ have a common
tangible root. \pSkip

\item  Suppose $f = \sum_{i=0}^m \a_i \lm^i $ and $g = \lm^\nu  + b
$, for $b$ tangible. Now the resultant matrix has entries
$\one^\nu$ instead of $\one,$ so $|\res(f,g)|$ equals
$$\a_0^\nu  + \a_1^\nu b + \a_2^\nu b^2 +   \cdots + \a_{m-1}^\nu b^{m-1} + \a_m b^m,$$ which is a ghost iff the $\nu$-value of $b$ is at most
that of $\frac {\al_{m-1}}{\al_m},$ the root of $f$ with greatest
$\nu$-value. Again, the resultant is a ghost iff $f$ and $g$ have
a common tangible root.

\end{enumerate}

\end{example}

\begin{example}\label{specialcases}  Suppose $f = \lm^2 + a^\nu \lm + b$ and  $g = \lm^2 + c^\nu \lm + d
$ are   quadratic quasi-essential polynomials over a supertropical
semifield; i.e.,
\begin{equation}\label{eq:exm1.1} \nu(a^2)
\ge b^\nu \quad \text{ and} \quad \nu (c^2) \ge d^\nu .
\end{equation} Accordingly
 \begin{equation}\label{eq:exm1.2}
Z_\tan(f) = \{ x \in \tT \ | \ (b / a)^\nu  \leq x^\nu \leq a^\nu
\} \quad \text{ and} \quad Z_\tan(g) = \{ x \in \tT \ | \ (d /c )
^\nu \leq x^\nu \leq c^\nu \}.
 \end{equation}

 The resultant $|\res(f,g)|$ of $f$ and $g$ is given
 by:

$$ |\res(f,g)| = \left|  \begin{array}{cccc}
             b & a^\nu &  \one &   \\
                      & b & a^\nu &  \one  \\
                     d & c^\nu & \one &    \\
                      &  d & c^\nu & \one
                  \end{array} \right| =d^2 + (acd)^\nu + (bd)^\nu +
                  (bc^2)^\nu + (a^2d)^\nu + (abc)^\nu +
                  b^2,$$
                  whose essential part, by \eqref{eq:exm1.1},  is
                  \begin{equation}\label{twoquads} d^2 + (acd)^\nu  +
                  (bc^2)^\nu + (a^2d)^\nu + (abc)^\nu +
                  b^2 = b f(c^\nu) + dg(a^\nu).\end{equation}
We show that $f$ and $g$ have no common tangible roots iff
$|\res(f,g)| \in \tT$, in which case obviously $|\res(f,g)| = b^2
+ d^2$.

$(\Rightarrow)$ Suppose $Z_\tan(f) \cap Z_\tan(g) = \emptyset$.
Thus, $b^\nu
> (ac)^\nu$ or  $d^\nu   > (ac)^\nu$; by symmetry, we may assume the first case, that $b^\nu
> (ac)^\nu$. Then
$$(a^2c)^\nu >  (bc)^\nu
> (ac^2)^\nu  > ( ad) ^\nu, $$ yielding $(ac)^\nu > d^\nu$
and thus
$$  (b^2)^\nu > (a^2c^2)^\nu > (a^2d)^\nu.$$
 Also $a^\nu \ge (\frac b a) ^\nu > c^\nu$ implies $$ (b ^2)^\nu > (abc)^\nu > (bc^2)^\nu;$$
 finally,
$$ (a^2 d)^\nu >  (bd)^\nu > (acd)^\nu >(d^2)^\nu,$$ yielding altogether  $|\res(f,g)| =
b^2$.

 $(\Leftarrow)$ Suppose that $|\res(f,g)|$ is tangible; then  $|\res(f,g)| =
 b^2$ or $|\res(f,g)| = d^2$. Assuming the former, we have $(b^2)^\nu  > (abc)^\nu
 $;
 i.e., $b^\nu  > (ac)^\nu $, implying $ Z_\tan(f) \cap Z_\tan(g) =
 \emptyset$.

 For intuition and future reference, we claim that
 the $\nu$-value of
 \eqref{twoquads} equals that of \begin{equation}\label{twoquads1}(a+c)\(a +\frac dc\)\(\frac ba + c\)
 \(\frac ba +\frac dc\).\end{equation}

 Indeed, by symmetry we may assume that $a^\nu \ge c^\nu.$ But \eqref{twoquads1} has the same $\nu$-value as
 $$
\begin{array}{lll}
 f(c)f(\frac dc) &= & (c^2+ a^\nu c +b)\((\frac dc)^2 + a^\nu\frac dc
 + b\) \\[1mm]
 &=& d^2 + a^\nu cd  +bc^2 + \frac{a^\nu d^2}c + (a^2)^\nu d +
 a^\nu cb + b(\frac dc)^2 + a^\nu b\frac dc + b^2,\end{array}$$
 which matches \eqref{twoquads} except for the extra terms $ \frac{a^\nu d^2}c , b(\frac dc)^2,$ and
 $a^\nu b\frac dc,$
which are dominated respectively by $d^2$, $bd$, and $a^\nu b c$;
  $bd$ is dominated in turn by $b^2 + d^2.$ \pSkip \end{example}

Although the formula for the \permanent\   is somewhat formidable,
and is quite intricate even for quadratic polynomials, it becomes
much simpler when
 the resultant is tangible, so our strategy is to reduce computations of the resultant to the tangible case as quickly as possible.

These examples indicate that the resultant is a ghost iff the
polynomials $f$ and $g$ have a common root.  The proof of this
result involves an inductive argument, which we prepare with some
notation. Given a polynomial $f = \sum _{i=0}^m \a_i \la ^i$, we
define
$$\spol{f}{\ell} = \sum _{i=\ell}^m \a_i \la ^{i-\ell}, \qquad \ell =
1,\dots, m;$$ thus, $f = \la \spol{f}{1} + \a_0 = \la^2
\spol{f}{2} + \a_1 \lm + \a_0 = \cdots.$ Recall from \cite[Lemma
7.28]{IzhakianRowen2007SuperTropical} that when $\a_1$ is
tangible, the polynomial $f = \sum _{i=0}^m \a_i \la ^i$ can be
factored as $(\la + \frac{\a_0}{\a_1})\spol{f}{1}$.

\begin{lemma}\label{prodres} If $f =  \sum _{i=0}^m \a_i \la ^i$ and $g =  \sum _{j=0}^n \bt_j \la ^j$ are full polynomials, then
$$|\res(f,g)| = \al_0 |\res(f,\spol{g}{1})| + \bt_0 |\res(\spol{f}{1},g)|.$$
\end{lemma}
\begin{proof} We expand the resultant $$ |\res(f,g)| = \left|
\begin{array}{ccccccc}
             \al_0  & \al_1 & \al_2 & \dots &  &  &  \\
            & \al_0  & \al_1 & \al_2 & \dots &   &   \\
              & & \al_0  & \al_1 & \al_2 & \dots   &   \\
                &   &  &  \ddots  &\dots &\ddots &  \vdots \\
                     \bt_0 & \bt_1 & \bt_2 & \dots & \bt_n &  &  \\
                   &  \bt_0 & \bt_1 & \bt_2 & \dots & \bt_n &     \\
                    &   &  \bt_0 & \bt_1 & \bt_2 & \dots  & \vdots     \\
                       &&      &\ddots &\dots &
                  \end{array} \right|$$ along the first column, to get
\begin{equation}\label{sumdets} \al_0 \left|  \begin{array}{ccccccc}
            \al_0  & \al_1 & \al_2 & \dots &  &  & \\
            & \al_0  & \al_1 & \al_2 & \dots &   \\
              &  &   \ddots  &\dots &\ddots & \dots \\
                    \bt_1 & \bt_2 & \dots & \bt_n &  &  \\
                    \bt_0 & \bt_1 & \bt_2 & \dots & \bt_n  &     \\
                      & \ddots & \dots    &\ddots & & \ddots
                  \end{array} \right| + \bt_0 \left|  \begin{array}{cccccc}
            \al_1 & \al_2 & \dots &  &  &  \\
              \al_0  & \al_1 & \al_2 & \dots &    &   \\
              & \ddots &  \ddots  &\dots &\dots &  \dots\\
                     \bt_0 & \bt_1   & \dots & \bt_n &   & \\
                   &  \bt_0 & \bt_1 & \dots & \bt_n &    \\
                      &   &  \ddots  &\dots &\ddots &  \ddots
                  \end{array} \right|
                  \end{equation}
                  In computing the second \permanent\   of
~\eqref{sumdets} by expanding along the first column,  the
occurrence of $\a_0$ in the second row must be multiplied by some
$\a_j$ in the first row, whereas, switching the first two rows, we
also have $\a_1 \a_{j-1}$, which has greater $\nu$-value than
$\a_0 \a_j$ since $f$ is essential. (Strictly speaking, we must
also consider the possibility that $(\a_0 \a_j)^\nu =(\a_1
\a_{j-1})^\nu, $ but in this case $\a_1$ must be ghost, so again
the term with $\a_0 \a_j$ is not relevant to the computation of
the \permanent.) Thus the occurrence of $\a_0$ in the second row
cannot contribute to the second \permanent\ of
Equation~\eqref{sumdets}, and we may erase it. Likewise, each
occurrence of $\al_0$ does not contribute to the second
\permanent\ of Equation~\eqref{sumdets}.

 By the same token, each occurrence of $\bt_0$ does not contribute to  the first \permanent\   of
Equation~\eqref{sumdets}. Thus, \eqref{sumdets} equals
\begin{equation*}\label{sumdets1}
\begin{array}{l}
 \al_0 \left|  \begin{array}{cccccc}
         \al_0 &   \al_1  & \al_2 &  \dots &  &  \\
        &  \al_0    & \al_1  &  & \dots &   \\
               &     &\ddots &\dots & &\\
                    \bt_1 & \bt_2 & \dots & \bt_n &  &    \\
                      & \bt_1 & \bt_2 & \dots & \bt_n  &     \\
                      &      &\ddots &\dots & &
                  \end{array} \right| + \bt_0 \left|  \begin{array}{cccccc}
            \al_1 & \al_2 &\dots  &  \\
            & \al_1 & \al_2 & \dots    \\
               &     &\ddots &\dots  &\\
                    \bt_0 & \bt_1 & \dots & \bt_n &  &    \\
                      & \bt_0 & \bt_1 & \dots &      \\
                      &      &\ddots &\dots &
                  \end{array} \right|
 \\
 \\[2mm]
                   = \al_0 |\res(f,\spol{g}{1})| + \bt_0
                  |\res(\spol{f}{1},g)|.
                  \end{array}
                  \end{equation*}
\end{proof}

\begin{lemma}\label{resm2} If $f =  \sum _{i=0}^m \a_i \la ^i$ is
a   full polynomial and $g =  \lm^2 + \bt_1 \lm + \bt_0$  is
irreducible with $(\bt_1^2)^\nu  \geq \bt_0^\nu$,  then
$$|\res(f,g)| = \sum_{\ell = 0}^{m-1} \al_\ell \bt_0^\ell \spol{f}{\ell}(\bt_1) +
    \bt_0^m g(\al_{m-1}).$$
\end{lemma}

\begin{proof} By definition   $\spol{g}{1} = \lm + \bt_1$,
and thus $|\res(\spol{f}{\ell}, \spol{g}{1})| =
\spol{f}{\ell}(\bt_1)$ for each $\ell = 0 , \dots, m-1$; cf.~
Example~\ref{specialcases} and \ref{specialcases2}. Use Lemma
\ref{prodres} recursively to write
$$
\begin{array}{lll}
  |\res(f,g)| &  =  & \al_0|\res(f,\spol{g}{1})| + \bt_0 |\res(\spol{f}{1},g)|
  \\[1mm]
    & = & \al_0 f(\bt_1)  +  \al_1 \bt_0 |\res(\spol{f}{1},\spol{g}{1})| + \bt_0^2 |\res(\spol{f}{2},g)|
    \\[1mm]
    & = & \al_0 f(\bt_1)  +  \al_1 \bt_0 \spol{f}{1}(\bt_1) + \bt_0^2 |\res(\spol{f}{2},g)|
    \\[1mm]
    & = & \al_0 f(\bt_1)  +  \al_1 \bt_0 \spol{f}{1}(\bt_1) + \al_2 \bt_0^2
    |\res(\spol{f}{2},\spol{g}{1})|+
    \bt_0^3|\res(\spol{f}{3},g)|
    \\[1mm]
    & = & \dots \\[1mm]
    & = & \sum_{\ell = 0}^{m-1} \al_\ell \bt_0^\ell \spol{f}{\ell}(\bt_1) +
    \bt_0^m |\res(\spol{f}{m-1},g)| \\[1mm]
    & = & \sum_{\ell = 0}^{m-1} \al_\ell \bt_0^\ell \spol{f}{\ell}(\bt_1) +
    \bt_0^m g(\al_{m-1}).
\end{array}
 $$
\end{proof}

\begin{rem}\label{canonform} We quote \cite[Proposition
7.28]{IzhakianRowen2007SuperTropical}:
 Suppose $f = \sum _j \a_j \la ^j \in F [
\la]$ is full.
 If   $\a_i \lm^i$ is a tangible essential monomial of $f$, then
 \begin{equation}\label{tanbreak}
 f
=
 (\al_t\lm^{t-i} + \al_{t-1}\lm^{t-i-1} + \cdots +
\al_{i+1}\lm + \a_i)\left(\lm^i + \frac{\al_{i-1}}{\a_i}\lm^{i-1}
+ \cdots + \frac{\al_0}{\a_i}\right).\end{equation} In the
notation of this paper,
 \begin{equation}\label{tanbreak15}
 f
= \left(\lm^i + \frac{\al_{i-1}}{\a_i}\lm^{i-1} + \cdots +
\frac{\al_0}{\a_i}\right)\spol{f}{i}.\end{equation}
\end{rem}
 We are ready for a
formula for the resultant. It is convenient to start with the
tangible case, both because it is more straightforward and also it
helps in tackling the general case.

\begin{theorem}\label{mainthm} Suppose that $f= \sum _{i=0}^m \al_i \la^i$ and $g= \sum _{j=0}^n \bt_j \la^j$ are both full polynomials over a supertropical semifield $F$,
where the $\al_i, \bt_j \ne \fzero$.
\begin{enumerate} \eroman
    \item If all the $\a_i, \bt_j$ are tangible, and $a_i = \frac {\al_{i-1}}{\al_i}$
    and $b_i = \frac {\bt_{i-1}}{\bt_i}$, then \begin{equation}\label{taneq01}|\res(f,g)| = \a_m^n \bt_n^m
    \prod _{i=1}^{m} \prod _{j=1}^{n} (a _i +
b_j) = \a_m^n \bt_n^m \prod _{i,j} |\res(\la +a_i,\la
+b_j)|.\end{equation}
   \item In general, take tangible $a_i$ such that  $({\a_{i}} a_i)^\nu =
 (\a_{i-1}) ^\nu$ and $b_j$ such that $({\bt_{j}} b_j)^\nu =
{(\bt_{j-1})} ^\nu$, for $0 \le i < m,$ $0 \le j < n.$ Then
$$|\res(f,g)|^\nu = \a_m^n \bt_n^m \prod _{i=1}^{m} \prod _{j=1}^{n} (a _i +
b_j)^\nu.$$ \item Notation as in (ii) and Lemma~\ref{prodres}, if
$a_1^\nu > b_1^\nu,$ then
\begin{equation}\label{makeeq1}|\res(f,g)|=
\al_0|\res(f,\spol{g}{1})|.\end{equation}
 \item For any polynomials $f,g,$ and $h$,
\begin{equation}\label{taneq0}|\res(f,gh)|^\nu = |\res(f,g)|^\nu
|\res(f,h)|^\nu \quad \text{and } \quad |\res(fg,h)|^\nu =
|\res(f,h)|^\nu |\res(g,h)|^\nu .\end{equation}
\end{enumerate}
\end{theorem}

\begin{proof}
(i) Noting that $ \spol{g}{1} = \sum_{j=1}^n \bt_j\la^{j-1}$ and
$b_1 = \frac{\bt_0}{\bt_1},$, we have  $$g =  \spol{g}{1}h$$ where
$h = \la + b_1$; in particular $\bt_0 = \bt_1 b_1$. Likewise, we
have  $f = (\la + a _1)\spol{f}{1}$. Also Lemma~\ref{prodres}
yields
\begin{equation}\label{makeeq}|\res(f,g)|= \al_0|\res(f,\spol{g}{1})|+
\beta_0 |\res(\spol{f}{1},g)|.\end{equation} By hypothesis that
$f$ is full, we have   $a_1^\nu \le a_2^\nu \le a_3^\nu  \le
\cdots$.

 Our strategy is to consider the remaining cases:
\begin{itemize}

    \item  $a_1^\nu \ne b_1^\nu$ in which case we want to show that one of the terms on the right
 side dominates the other, and equals $|\res(f,\spol{g}{1})| |\res(f,h)| $
 (and thus also equals $|\res(f,g)|$ by bipotence). \pSkip

    \item $a_1^\nu = b_1^\nu$, in which
case we want to show that both of the terms on the right side of
Equation~\eqref{makeeq} have $\nu$-value equal to
$|\res(f,\spol{g}{1})|^\nu |\res(f,h)|^\nu $, whereby
$|\res(f,g)|$ is ghost and equal to $|\res(f,\spol{g}{1})|
|\res(f,h)| $.\pSkip

\end{itemize}






  If $a_1^\nu > b_1^\nu$,  then $b_1^\nu < a_1^\nu \le a_i^\nu $,
  implying $(b_1^{i})^\nu < (a_i b_1^{i-1})^\nu   ,$ and thus
  $(\al_i b_1^i)^\nu < (\al_{i-1} b_1^{i-1})^\nu < \cdots \le (\al_1 b_1)^\nu < (\al_0)^\nu$. Thus, by bipotence,
   $\al_0 =
 f(b_1) = |\res(f,h)|$, so the first term of the right side of \eqref{makeeq} is
 $$\a_0|\res(f,\spol{g}{1} )| =  |\res(f,h)||\res(f,\spol{g}{1} )|,$$
 which equals the right side of  Equation \eqref{taneq01} by
 induction.
 Hence, to prove $|\res(f,g)| = |\res(f,\spol{g}{1})| |\res(f,h)|, $
  we need only show that $\bt_0 |\res(\spol{f}{1},g)|$
has $\nu$-value $< |\res(f,\spol{g}{1})|
  |\res(f,h)|$. (This also proves (iii) for tangible polynomials.)
 By induction on $m$,   $|\res(\spol{f}{1},g)| =
|\res(\spol{f}{1},\spol{g}{1})||\res(\spol{f}{1},h)|.$ By
Lemma~\ref{prodres}, $ \bt_1|\res(\spol{f}{1},\spol{g}{1})|$ has
$\nu$-value $\le |\res(\spol{f}{1},g)|^\nu.$ But
$$b_1|\res(\spol{f}{1},h)|^\nu = b_1 (\spol{f}{1}(b_1))^\nu <  f(b_1)^\nu =  |\res(f,h)|^\nu
,$$ so multiplying together (noting that $\beta _0 = \beta _1
b_1$), we see that
$$\bt_0
 |\res(\spol{f}{1},g)|^\nu = \bt_1|\res(\spol{f}{1},\spol{g}{1})|\, b_1|\res(\spol{f}{1},h)|^\nu < |\res(f,\spol{g}{1})|
 |\res(f,h)|^\nu,$$ as desired.

 We  want to conclude that
\begin{equation}\label{firstvers}|\res(f,g)| = \al_m^{n} \bt_n^m
\prod _{i=1}^{m} (a _i + b_1) \prod _{j=2}^{n} (a _i + b_j)
.\end{equation} Note that $|\res(f,h)| = f(b_1) = \a_m
 \prod (a_i +b_1)$, whereas, by induction,  $$|\res(f,\spol{g}{1})|  =
 \a_m^{n-1} \bt_n^m \prod _{i=1}^{m}  \prod
_{j=2}^{n} (a _i + b_j);$$  we  get \eqref{firstvers} by
multiplying these together.

If $a_1^\nu < b_1^\nu$, then $b_1 = h(a_1),$ and thus the second
term of the right side of \eqref{makeeq} is  $$\bt_0
|\res(\spol{f}{1},g)| = h(a_1)\bt_1|\res(\spol{f}{1},g)|,$$ and we
get \eqref{firstvers} by the same induction argument (applied now
to the left side).

Finally, if $a_1^\nu =b_1^\nu$, then the same argument shows that
the two terms on the right side of Equation~\eqref{makeeq} are
both $\nu$-matched to   the right side of  Equation
\eqref{taneq01}, implying that the right side of
Equation~\eqref{makeeq} is ghost, and it remains to show that the
left side is also ghost. But this is clear since the assumption
$a_1^\nu =b_1^\nu$ implies that $a_1$ is a common root of $f$ and
$g$. Thus, we have verified~\eqref{firstvers}, yielding (i); we
also have obtained (iii) for tangible polynomials.

(ii) and (iii) follow, since we can replace the $\a_i$ and $\bt_j$
by tangible coefficients of the same $\nu$-value.

(iv) follows for the same reason, since once we replace the
coefficients of $f,$ $g,$ and $h$ by tangible coefficients of the
same $\nu$-value, we may factor them further and apply (i).
\end{proof}

\begin{corollary}\label{prop:4b12}  Suppose $f = \prod _{i=1}^m (\la + a_i)$
and $g = \prod _{j=1}^n (\la + b_j)$ are tangible.
 Then
$$\dtr{{\res(f,g)} } = \prod _{i,j} ( a _i + b_j) =  \prod
_j f(b_j)=\prod_i g(a_i).$$ \end{corollary}

We turn to full polynomials over a supertropical semifield $F$
(with $\tT = \tT(F))$, recalling their decomposition
 from \cite[Section 7]{IzhakianRowen2007SuperTropical}.

\begin{defn} A full polynomial $f = \sum _{i=0}^t \al _i \la
^i$ is \textbf{semitangibly-full}  if  $\al_t$ and $\al _0$ are
tangible, but $ \al _i$ are ghost for all $0<i<t$; $f $ is
\textbf{left semitangibly-full}, (resp.~\textbf{right
semitangibly-full}) if $\al _0$ is tangible and $\al_i$ are ghost
for all $0<i \le t$ (resp. ~$\al _t$ is tangible and $\al_i$ are
ghost for all $0\le i < t$).
\end{defn}

\begin{rem} Suppose $f = \sum _{i=0}^t \al _i \la
^i$ is a full polynomial. Then,  taking $a_1 = \frac {\hat
\a_0}{\hat \a_1}\in \tT$  and $a_{t}= \frac {\hat \a_{t-1}}{\hat
\a_t}\in \tT,$ we have (in logarithmic notation)
\begin{equation}
Z_{\operatorname{tan}}(f) = \begin{cases} [a_1,a_t] \text{ for } f \text{ semitangibly-full};\\
[a_1,\infty) \text{ for } f \text{ left semitangibly-full};\\
(-\infty,a_t] \text{ for } f \text{ right semitangibly-full},
\end{cases}
\end{equation}
as seen by inspection. Indeed, the $\nu$-smallest tangible root of
$f$ is $a_1$ when $\a_0$ is tangible (since for any $a$ of
$\nu$-value $< a_1^\nu,$ one has $f(a) = \a_0$). Likewise, the
$\nu$-largest tangible root of $f$ is $a_t$ when $\a_t$ is
tangible.

Put in the terminology of Definition \ref{lefthalf}, if $f$ is
left semitangibly-full, then $f$ is $a_1^\nu$-left
 half-tangible; if   $f$ is
right semitangibly-full, then $f$ is $a_t^\nu$-right
 half-tangible.
\end{rem}

Equation \eqref{tanbreak} shows that we can always factor $f$ at
tangible essential monomials into factors with disjoint root sets,
leading immediately to the following assertion \cite[Proposition
7.36]{IzhakianRowen2007SuperTropical}:

\begin{prop}\label{canonform1} Any full polynomial $f$ can be decomposed as a product
\begin{equation}\label{tanbreak1} f = \fr f_1  \cdots
f_t \fl \end{equation} where the   polynomial $\fr$ is
semitangibly-full or left semitangibly-full, $f_1, \dots, f_t,$
are semitangibly-full,
 and $\fl$ is semitangibly-full or right semitangibly-full polynomial, and their tangible root
sets are mutually disjoint intervals with descending $\nu$-values.
\end{prop}

\begin{rem}\label{canonform2} Equation \eqref{tanbreak15} implies that in this
decomposition \eqref{tanbreak1}, $$\spol{f}{i} =  \fr f_1 \cdots
f_t {\spol{f}{i}}^\lstf$$ for each $i \le \deg (\fl).$\end{rem}

 Thus it makes sense for us to compute the resultant  of
 semitangibly-full polynomials.

\begin{lem}\label{mainthm0} Suppose that $f= \sum _{i=0}^m \al_i \la^i$, $g,$ and $h= \sum _{j=0}^n \bt_j
\la^j$ $(n \ge 1)$ are polynomials whose root sets are disjoint
intervals;   assume that all roots of $f$ and $g$ have $\nu$-value
greater than every tangible root of $h$.  Then
$$|\res(f,gh)| = \al_0^{n}\bt_n^m |\res(f,g)|,$$
which is tangible iff $|\res(f,g)|$ is tangible. Explicitly, for
each $i \le n,$
$$|\res(f,gh)| = \al_0^{i}|\res(f,\spol{(gh)}{i})|.$$
\end{lem}
\begin{proof} If $n=1$ then the assertion is clear from
Example~\ref{specialcases2}(i), so we assume that $n>1.$ By
Theorem~\ref{mainthm}(iii),
$$|\res(f,gh)| = \al_0 |\res(f,\spol{(gh)}{1}).$$
But the same argument shows that $|\res(f,\spol{(gh)}{1})|= \al_0
|\res(f,\spol{(gh)}{2})|,$  and we have
$$|\res(f,gh)| = \al_0^2 |\res(f,\spol{(gh)}{2})|.$$
Iterating, after $i$ steps we get $$|\res(f,gh)| = \al_0^{i}
|\res(f,\spol{(gh)}{i})|;$$ taking  $i =n$ yields
$$|\res(f,gh)| = \al_0^{n} \bt_n^m
|\res(f,g)|,$$ as desired.
  \end{proof}

Note that in Lemma~\ref{mainthm0}, $f$ could be either left
semitangibly-full or  semitangibly-full, and $h$ could be either
semitangibly-full or right semitangibly-full, but in every case
the result is the same.

\begin{theorem}\label{mainthm2} Suppose that $f= \sum _{i=0}^m \al_i \la^i$ and $g= \sum _{j=0}^n \bt_j \la^j$ are both full polynomials over a supertropical semifield $F$,
where the $\al_i, \bt_j \ne \fzero$.
\begin{enumerate} \eroman
    \item Suppose $f$
and $g$ are full polynomials, decomposed as products  as in
Proposition~\ref{canonform1}; i.e.,  $$f = \fr f_1  \cdots f_t
\fl, \qquad g = \grt g_1  \cdots g_u \gl,$$ and suppose the
tangible root sets of $f$ and $g$ are disjoint. Let $f_0 = \fr$,
$g_0 = \grt,$ $f_{t+1} = \fl$, and $g_{u+1} = \gl.$ Then
$$|\res(f,g)| = \prod _{j=0}^{t+1}\prod _{k=0}^{u+1} |\res(f_j,g_k)|, $$ each of which can
be calculated according to Lemma~\ref{mainthm0}.

\item  Specifically, if $|\res(f,g)|$ is tangible and $g = \prod_j
(\la + b_j),$ then $|\res(f,g)| = \prod _j f(b_j\textbf{}).$
\medskip
\item  $|\res(f,g)|$ is ghost iff $f$ and $g$ have a common root.
\end{enumerate}
\end{theorem}
\begin{proof} $ $
\begin{enumerate} \eroman
    \item We may assume that  every tangible root of $\gl$ has
$\nu$-value less than every tangible root of $f$ as well as every
tangible root of $\grt g_1 \cdots g_u$. Let $n_u = \deg (\gl).$
Remark~\ref{canonform2} applied to Lemma~\ref{mainthm0} implies
$$
\begin{array}{lllllll}
  |\res(f,g)| & = &  \al_0^{i}|\res(f,\spol{g}{n_u})| & = &  \al_0^{i}|\res(f,
\grt g_1  \cdots g_u  \spol{\gl}{n_u})| \\[1mm]
  & &  & = & \al_0^{n_u}\bt_n^m |\res(f, \grt g_1  \cdots g_u | \\[1mm] & & &  =
  &
|\res(f,\gl)|\res(f, \grt g_1  \cdots g_u )|, \\
\end{array}
$$ and one continues by induction on $t+u$.
 \pSkip
    \item Follows from (i). \pSkip

    \item  Follows from Remark \ref{easydir}  and the contra positive of
(i). If $|\res(f,g)|$ is ghost, and the root sets are disjoint,
some $|\res(f_j,g_k)|$ must be ghost, contradicting
Lemma~\ref{mainthm0}.
\end{enumerate}
\end{proof}

Putting together Theorems \ref{thm:rootPrime} and \ref{mainthm2}
yield our main result:
\begin{theorem}\label{thm:resOfTan}
Polynomials $f = \sum \al _i \la ^i$ and $ \ g= \sum \bt_j \la ^j$
satisfy $\dtr{ \res (f,g) } \in \tGz$ iff $f$ and $g$ are not
relatively prime, iff $f$ and $g$ have a common tangible root.
\end{theorem}

The next example illustrates the assertion of Theorem
\ref{thm:resOfTan}.

\begin{example} Let $$f = (\lm+a)(\lm+b) = \lm^2 + b\lm + ab  \qquad \text{ and } \qquad  g= \lm +
c,$$ where $a,b, c \in R$ and $b^\nu > a^\nu$.  Then
\begin{equation}\label{exam1} \res(f,g) = \dtr{
\begin{array}{ccc}
 ab &  b &  \one \\
c &   \one &  \\
& c    & \one \\
\end{array}%
} = c^2  + b c + a b. \end{equation}
\begin{enumerate}
    \item If   $f$ and $g$ have a common tangible root, that is  $a$
    or $b$ respectively, then
$c^\nu  =  a^\nu$ (resp.~ $ c^\nu  = \ b^\nu$),
    and   clearly  $\res(f,g) = (ab)^\nu$ (resp.~    $\res(f,g) = (b^2)^\nu$) is
    ghost. \pSkip

    \item When  $c ^\nu \neq a ^\nu , b ^\nu $, and thus $f$ and $g$ have no common factor,
    and $\res(f,g) \in \tGz,
    $ then  at least one term in \eqref{exam1} is ghost. As usual, we
    take tangible
    $\hat a, \hat b, \hat c$  such that  $(\hat a)^\nu =
    a^\nu, $  $(\hat b)^\nu =
    b^\nu, $ and  $(\hat c)^\nu =
    c^\nu. $

    Assume first that $\res(f,g) =
    (ab)^\nu$ and $a$ or $b$ is ghost. Then $ c ^\nu \leq
    a ^\nu$ and thus $\hat c$ is also a root of $f$.

If $\res(f,g) =
    (bc)^\nu$, then  $a^\nu < c^\nu < b^\nu.$ If $b$ is ghost, then   $\hat c$ is a
common root of $f$ and $g$. But if $b$ is tangible, then $c$ is
ghost and $\hat a$ is a common root of $f$ and $g$.

Finally, if $\res(f,g) = c^2$, where $c$ is ghost, then  $ a ^\nu,
b^\nu \le c^\nu $, so $\hat a$ and $\hat b$ are common tangible
roots of $f$ and $g$.
\end{enumerate}

\end{example}

\subsection{A second proof of Corollary \ref{prop:4b12} using the generic method}

Since Corollary \ref{prop:4b12}  encapsulates the basic property
of the resultant, let us present a second proof using a different
approach of independent interest. We start with a multivariate
version of \cite[Lemma~7.6]{IzhakianRowen2007SuperTropical}.

\begin{lemma}\label{rootd}
Suppose $f \in F[\la_1, \dots, \la_n]$, and let $f^a( \la_1,
\dots, \la_{k-1}, \la_{k+1}, \dots, \la _n)$ denote the
specialization of $f$ under $\la_k \mapsto a\in \tT.$ Suppose that
the polynomial $f^a$ becomes ghost on a nonempty open interval
$\oset_a$ of $F^{(n-1)}$  and also assume that every tangible open
interval $W_\tT$ of $F^{(n)}$ contains some point $\bfb$ that
$f(\bfb) \notin \tGz.$ Then $(\la_k + a)$ e-divides~$f$.
\end{lemma}

\begin{proof} By symmetry of notation, one may assume $k=n.$

 The case $n =1$ is just \cite[Lemma~7.6]{IzhakianRowen2007SuperTropical}, since the
assertion is that $f  \in  F[\la_1] $ becomes a constant that is a
ghost on an open interval, and thus is a ghost; i.e., $a$ is a
root of $f$ which by assumption is ordinary. Thus, we may assume
$n>1.$ Write $f = \sum_i f_i \la_1^i,$ for $f_i \in F[\la_2,
\dots, \la_{n}].$   In particular, $f^a = \sum_i f_i^a \la_1^i.$
But the $f_i^a \la_1^i$ are $\nu$-distinct on some open interval
since they involve different powers of $\la_1$. Hence, by Lemma
~\ref{helpgh}, each $f_i^a \la_1^i$ (and thus $f_i$) is ghost on
some nonempty open interval, so by induction $(\la_n+a)$ e-divides
each $f_i$, and thus also $e$-divides $f$.
\end{proof}

\begin{proof}[Second proof of Corollary~\ref{prop:4b12}:] Using the
generic method, we consider the case where all the $a_i$ are
(tangible) indeterminates over the supertropical semifield $F$;
then $\dtr {\res(f,g)} $ is some polynomial in $F[a_1, \dots,
a_m].$ But, substituting $a_i \mapsto b _j$ yields a common root
for $f$ and $g$, and thus sends $\dtr{ {\res(f,g)}} $ to $\tGz$,
in view of Theorem \ref{thm:resOfTan}. Hence, by Lemma
~\ref{rootd}, $( a_i +b _j)$ e-divides $\dtr {\res(f,g)} $ for
each $i,j,$ and these are all distinct, implying by an easy
induction argument that $\prod _{i, j}(a_i +b _j)$ e-divides $\dtr
{\res(f,g)}$.

Clearly, $\prod_{i, j}(  a _i + b_j )$ has degree $mn$. So let us
compute the degree of $\dtr {\res(f,g)}$. For $i\le n,$ the
$(i,j)$ term in $\res(f,g)$ (when nonzero) has degree $m+j-i.$ For
$i> n,$ the $(i,j)$ term in $\res(f,g)$ (when nonzero) has
 degree 0. Thus it follows from the
formula for calculating the \permanent\   that $|\res(f,g)|$ has
degree $mn + \sum j - \sum i = mn.$
 One concludes $\dtr {\res(f,g)} = c \prod_{i,j}( a _i + b_j   )$ for some $c \in F.$ But the term in
$\prod _{i,j} ( a _i + b_j )$ without any $b_j$ is precisely
$(a_1\dots a_m)^n,$ which occurred by itself in $\dtr
{\res(f,g)}$. Thus $c=\fone,$ proving the first assertion. The
other equalities follow at once. For example, $\prod _i (b _j + a
_i ) = f(b_j),$ so
$$\prod _{i,j} ( a _i + b_j) = \prod _j f(b_j).$$
\end{proof}

\section{B\'{e}zout's theorems}

One of the major applications of the resultant to geometry is
B\'{e}zout's theorem. Throughout this section we assume that $F$
is a supertropical semifield. Suppose $f,g$ in $F[\lm_1,\lm_2]$.
Rewriting  the polynomials in terms of $\tilde \lm = \lm_1/\lm_2$
and $\lm = \lm_2$, the polynomials $f$ and $g$ can be viewed as
polynomials in $\lm$, with coefficients in $F[\tilde \lm]$. From
this point of view, the resultant $\abs{\res(f,g)}$ is a
polynomial $p(\tilde \la).$

\begin{theorem}[\textbf{B\'{e}zout's theorem}]\label{thm:Bezout}
Nonconstant polynomials $f,g$ in $F[\lm_1,\lm_2]$ cannot have more
than $mn$ 2-ordinary  points in the intersection of their sets of
projective roots, where $m = \deg(f)$ and $n = \deg(g)$.
\end{theorem}

\begin{proof}
Assume that the tangible points $(x_i,y_i)$ lie on each root set,
$C_f$ and $C_g$, defined respectively by the roots of $f$ and $g$,
for $i = 1,\dots, mn+1$. After a suitable additive translation,
cf.~Remark~\ref{transf1}, we may assume that each $y_i \neq
\fzero$. Then, after a suitable Frobenius morphism
(Remark~\ref{Frobmor}), we may assume that the $\frac {x_i}{y_i}$
are distinct (as well as finite). Let $\tilde \la = \frac
{\la_1}{\la_2},$ and view $f,g$ as polynomials in $R[\la_2]$,
where
 $R = F[\tilde \la].$

 Viewing $\abs{\res(f,g)}$  in $R = F[\tilde \lm]$, one sees that for
any specialization $\tilde f$ and $\tilde g$  given by $\tilde \lm
\mapsto x_i/y_i$, $\tilde f$~and $\bar g$ have the common
2-ordinary root $y_i$,    and thus their resultant is ghost. In
other words $(\tilde \lm + x_i / y_i)$ $e$-divides
$\abs{\res(f,g)}$ for each $i = 1,\dots,mn+1$. Hence
$\deg\(\res(f,g)\) > mn+1$. But by definition $\abs{\res(f,g)}$
has degree $mn$ --  a contradiction.
\end{proof}

\section{Supertropical divisibility and the Nullstellensatz}

In \cite[Theorem 6.16]{IzhakianRowen2007SuperTropical} we proved a
Nullstellensatz involving tangible polynomials. In order to
formulate a more comprehensive version, we need a more general
notion of divisibility. Suppose $R$ is a semiring with ghosts,
which satisfies supertropicality (Note \ref{supertr}).

\begin{defn}\label{superdiv} An element $g\in R$ \textbf{supertropically divides} $f\in R$
if $f+qg$ is ghost with $(f+qg)^\nu = f^\nu$, for a suitable
    $q\in \tT$.
\end{defn}

\begin{example} Let $R = F[\lambda]$, and  consider the polynomial $f = \la^2 + 6 ^\nu \la +7$, whose tangible root set is the
interval $[1,6]$.
\begin{enumerate} \eroman
    \item If $g = \la  + 4,$ whose tangible root set is $\{ 4 \},$ then
$$f + (\la + 3)g = f + \la^2 + 4 \la +7$$ is ghost. \pSkip

    \item $g = \la^2 + 4 ^\nu \la +6$, whose tangible root set is the
interval $[2,4]$, then $$f^2 +(\la^2+8)g  =  (\la^4 + 6 ^\nu \la^3
+ 12^\nu \la^2 +13^\nu \la + 14)+(\la^4 +4^\nu \la^3 + 8 \la^2 +
12^\nu \la + 14),$$ which is ghost.
\end{enumerate}

\end{example}

\begin{example}\label{quad1} Suppose $f = \la^2 + a_2 ^\nu \la
+a_1a_2,$ for $a_1, a_2$ tangible. Then, for $a$ tangible, $\la
+a$ supertropically divides $f$ iff $a_1 ^\nu \le a ^\nu \le a_2
^\nu.$ Indeed, for the constant term of $f+(\la+a)q$ to be ghost,
we must have $(aq)^\nu = (a_1a_2)^\nu.$ The coefficient of $\la$
shows that $\max\{q^\nu, a^\nu\} < a_2^\nu,$ and thus
$\min\{q^\nu, a^\nu\} < a_1^\nu.$\end{example}

\begin{defn}  Suppose $A \subset R $.
 The \textbf{supertropical radical} $\tsqrt{A}$ is defined as
 the set $$\{ a \in R : \text{ some power}  \quad a^k\quad   \text{supertropically divides
 an element of } A \},$$
  which in other words is
  $$\{ a \in R : (a^k+b)^\nu = (a^k)^\nu \text{  and } a^k +b\in \tGz,
  \text{ for some } b\in A \text{ and some } k \in \Net ^+
  \}.$$
  An ideal $A$ of $R$ is \textbf{supertropically radical} if
$A = \tsqrt{A}.$
\end{defn}

\begin{rem}\label{surprise1} If $A$ is an ideal of a commutative semiring $R$, then
$\tsqrt A\triangleleft R$. Indeed, if $a_1^{k_1}+b_1 \in \tGz$ and
$a_2^{k_2}+b_2 \in \tGz$, then by the Frobenius map,
$$(a_1+a_2)^{k_1k_2} + b_1^{k_2} + b_2^{k_1} = a_1^{k_1k_2} + b_1^{k_2} +a_2^{k_1k_2} + b_2^{k_1}=  (a_1^{k_1 } + b_1)^{k_2} +(a_2^{ k_2} + b_2)^{k_1}
$$
which is ghost, of the same $\nu$-value as $a_1^{k_1k_2} +
a_2^{k_1k_2} = (a_1+a_2)^{k_1k_2}$. Likewise, for all $r$ in $R$,
$(ra_1)^{k_1}+r^{k_1}b_1 \in \tGz$, of the same $\nu$-value as
$(ra_1)^{k_1}$.

\end{rem}

By the same sort of argument, if $R$ is a commutative
supertropical \emph{semiring} and $A$ is a sub-semiring of
$\FunR$, then $\tsqrt A$ is also a sub-semiring of $\FunR$.

\subsection{The comprehensive supertropical Nullstellensatz}

The comprehensive supertropical version of the Hilbert
Nullstellensatz is as follows (with the same proof as in
\cite[Theorem 6.16]{IzhakianRowen2007SuperTropical}).

For a polynomial
 $f\in F[\lm_1, \dots, \lm_n],$ we define the set
$$D_f = \{ \bfa = (a_1, \dots, a_n) \in \tTz^{(n)} : f(\bfa) \in \tT  \};$$
thus  $\tTz^{(n)}\setminus D_f$ is the set of tangible roots of
$f$ in $\tTz^{(n)}.$  Refining this definition, writing $f = \sum
f_\bfi$, a sum of monomials,  $D_{f,\bfi}$  is defined to be
$$D_{f,\bfi} = \{\bfa = (a_1, \dots, a_n) \in \tTz^{(n)} : f(\bfa) = f_\bfi (\bfa)
\in \tT \}.$$ Therefore, $\bfa \in D_{f,\bfi}$ iff $f$ has the
monomial $f_\bfi$ with $f_\bfi(\bfa)$ tangible, and
$f_\bfi(\bfa)^\nu> f_\bfj(\bfa)^\nu$ for all $\bfj \ne \bfi;$
hence $D_f = \bigcup _{\bfi} D_{f,\bfi}.$

 Clearly, the $D_{f,\bfi}$ are open sets, and each has a finite number of
connected components~$D_{f,\bfi_u},$  for $1 \le u \le t= t_\bfi$,
called  the \textbf{irreducible components} of $f$. Note that
$D_f$ is the disjoint union of the~$D_{f,\bfi_u}$.

Given an irreducible component $D $ of $f$, we write $f \preceq _D
g$ if $g$ has an irreducible component containing~$D;$ $f \epsc S$
for $S \subseteq F[\lm_1,\dots,\lm_n],$ if for every irreducible
component $D = D_{f,\bfi_u}$ of $f$ there is some $g\in S$
(depending on $D_{f,\bfi_u}$) with $f \preceq _D g$.

We recall \cite[Definition 6.12]{IzhakianRowen2007SuperTropical}.

\begin{defn}\label{epst1}
Given an irreducible component $D $ of $f$, we write $f \preceq _D
g$ if $g$ has an irreducible component containing $D;$ $f \epsc S$
for $S \subseteq F[\Lambda],$ if for every irreducible component
$D = D_{f,\bfi_u}$ of $f$ there is some $g\in S$ (depending on
$D_{f,\bfi_u}$) with $f \preceq _D g$.

Also, define  the \textbf{dominant monomial} of $f$ on the
irreducible component $D$, denoted $f_D$, to be that monomial
$f_\bfi$ such that $f(\bfa) = f_\bfi(\bfa) \in \tT $ for every
$\bfa \in D.$
\end{defn}

\begin{thm}\label{Null1} \textbf{(Comprehensive supertropical Nullstellensatz)}
Suppose $F$ is a  connected, $\Net$-divisible, supertropical
semifield,  $A \triangleleft F[\la_1, \dots, \la _n],$ and $f \in
F[\la_1, \dots, \la _n].$   Then $f \epsc A$ iff $f \in
\tsqrt{A}.$
\end{thm}

\begin{proof} We review the proof in \cite[Theorem 6.16]{IzhakianRowen2007SuperTropical}. Again, the direction $(\Leftarrow)$ is clear, so we prove $(\Rightarrow)$. Let $\hat f$
denote the tangible polynomial having the same $\nu$-value as $f$.
Write $D_{\hat f}$ as the disjoint union of the irreducible
components $D_{\hat f,\bfi}$ of the complement set of
$Z_{\operatorname{tan}}(\hat f)$, which we number as $D_1, \dots,
D_q.$ Some of these remain as components of  the complement set of
$Z_{\operatorname{tan}} (f)$; we call these components ``true''.
Other components are roots of $f$ (because of its extra ghost
coefficients) and thus belong to $Z_{\operatorname{tan}} (f)$; and
 we call these components ``fictitious.'' For each true
component $D_k$ take a polynomial $g_k\in A$ with an irreducible
component $D'$ containing $D_k.$ Let $g_{k,\bfj}=
 \beta_\bfj \la_1^{j_1}\cdots \la_n^{j_n}$, $\bfj = (j_1,\dots, j_n)$, be the dominant
monomial of $g_k$ on $D'$.

 At any stage, we may replace $f$ by a
power $f^m$ (and, if necessary, $g_k$ by $g_k^m$ times some
element of $\tT$), for this does not affect its irreducible
components; we do this where convenient.

As in the proof in \cite[Theorem
6.16]{IzhakianRowen2007SuperTropical}, on each true $D_k$, $f_k=
g_{k,\bfj}$.
 But \cite[Lemma
6.14]{IzhakianRowen2007SuperTropical}, and induction on the number
of (possibly fictitious) components separating $D_k$ from the
other components still  shows that $f^{m'}$ dominates $g_k$ on
each  component. Now we apply the same argument to all true
neighbors, and continue until we have taken into account all of
the (finitely many) true components. The proof is then completed
by taking $ m
> \max_k \{ m_k + m'_k \}$; then $f^m = \sum_k g_k \in A$ on the
true components and $f^m$ dominates $ \sum_k g_k $ on the
fictitious components, implying $(f^m)^\nu = ( f^m +\sum_k
g_k)^\nu$ and $  f^m +\sum_k g_k$ is ghost, as desired.
\end{proof}

\subsection {B\'{e}zout's Theorem revisited}

Theorem \ref{thm:Bezout} probably can be generalized to the
non-tangible situation, but we do not yet have a full proof:

\begin{conjecture}
 [\textbf{Supertropical
B\'{e}zout conjecture}]\label{thm:Bezout2} Nonconstant polynomials
$f,g$ in $F[\lm_1,\lm_2]$  cannot have more than $mn$ tangible
connected components in the intersection of their sets of tangible
roots, where $m = \deg(f)$ and $n = \deg(g)$.
\end{conjecture}

Here is a proposed method to prove this conjecture. First, we need
a more comprehensive version of Lemma \ref{rootd}.

\begin{lemma}\label{rootd1}  Suppose $f \in F[\la_1, \dots, \la_n]$, and $W$ is a
component of $Z_{\operatorname{tan}}(f)$. If the projection of $W$
on the $k$ component is a closed interval $[b_1, b_2]$ where
$b_1^\nu < b_2^\nu,$ then for any $a$ with $b_1^\nu \le a^\nu \le
b_2^\nu]$, the polynomial $(\la_k +a)$ supertropically divides
$f$.
\end{lemma}

\begin{proof}
 The case $n =1$ follows from
\cite[Proposition~7.47]{IzhakianRowen2007SuperTropical}, since the
assertion is that for each element $a$ in the interval
$[b_1,b_2],$ the polynomial $f$ specializes to a constant that is
a ghost on a tangible open interval, and thus is a ghost; i.e.,
$a$ is a root of $f$ in $F[\la_1]$, in view of Example
\ref{quad1}. Thus, we may assume $n>1.$ Write $f = \sum_i f_i
\la_1^i,$ for $f_i \in F[\la_2, \dots, \la_{n}].$ In particular,
$f^a = \sum_i f_i^a \la_1^i.$ But the $f_i^a \la_1^i$ are
$\nu$-distinct on some open interval since they involve different
powers of $\la_1$. Hence, by Lemma ~\ref{helpgh}, each $f_i^a
\la_1^i$ (and thus $f_i$) is ghost on some nonempty open interval,
so by induction $(\la_k +a)$ supertropically divides each $f_i$,
and thus also supertropically divides~$f$.
\end{proof}

The difficulty in completing the proof along the lines of the
proof of Theorem~\ref{thm:Bezout} is that one cannot say for $\la
+a_i$  supertropically dividing $f$ for $1\le i \le t$ that also
$\prod _j (\la +a_j)$ supertropically divides $f$. For example,
this is false for $f = \la^2 + 6 ^\nu \la +7$, $a_1 = 2,$ $a_2 =
3,$ and $a_3 = 5$. The reason is that   these roots all lie on the
same connected component. Presumably, one may be able to complete
the proof by counting the number of connected components of the
complement set of the resultant, with respect to a suitable
projection onto the line.



\end{document}

%% file: tropMacro.tex
\usepackage{epsfig,latexsym,amsfonts,amssymb,amsmath,amscd,graphics,epic}
\usepackage{amsfonts,amssymb,amsmath,amscd,amsthm}
\usepackage[mathscr]{eucal}
\usepackage{mathrsfs}
\usepackage{oldgerm,units}
\usepackage{wrapfig,epsfig}

\usepackage{ifthen}
\usepackage{mathbbol}

\usepackage{amsthm}
\usepackage[mathscr]{eucal}
\usepackage{mathrsfs}
\usepackage{mathbbol}
\usepackage{oldgerm,units}
\usepackage{wrapfig}

\newtheorem{theorem}{Theorem}[section]
\newtheorem{proposition}[theorem]{Proposition}
\newtheorem{definition}[theorem]{Definition}

\newtheorem{lemma}[theorem]{Lemma}
\newtheorem{conjecture}[theorem]{Conjecture}

\newtheorem{corollary}[theorem]{Corollary}

\newtheorem{example}[theorem]{Example}
\newtheorem{remark}[theorem]{Remark}

\newtheorem{note}[theorem]{Note}

\newcommand{\Real}{\mathbb R}

\newcommand{\Net}{\mathbb N}

\newcommand{\one}{\mathbb{1}}
\newcommand{\zero}{\mathbb{0}}

\newcommand{\trop}[1]{\mathcal{#1}}

\newcommand{\tG}{\trop{G}}

\newcommand{\tT}{\trop{T}}

\newcommand{\To}{\longrightarrow }

\newcommand{\al}{\alpha}
\newcommand{\bt}{\beta}
\newcommand{\gm}{\gamma}
\newcommand{\Gm}{\Gamma}

\newcommand{\lm}{\lambda}

\newcommand{\abs}[1]{\left\vert#1\right\vert}

    \ifx\proof\undefined
    \newenvironment{proof}{
    \smallskip
    \noindent\emph{Proof.}}{\hfill\(\Box\)
    \bigskip
    } \fi

\newcommand{\ifdef}[3]{\ifthenelse{\equal{#1}{true}}{#2}{#3}}